\documentclass[psamsfonts, reqno, 11pt]{amsart}
\usepackage{fullpage,amssymb,stmaryrd}
\usepackage{lmodern}
\usepackage[T2A,T1]{fontenc}
\usepackage[utf8]{inputenc}
\usepackage[russian,english]{babel}
\usepackage{hyperref}
\usepackage[hyphenbreaks]{xurl}
\usepackage{amsmath}
\usepackage{xcolor}
\usepackage{xparse}
\usepackage{amsthm}
\usepackage{pgfplots}
\usepackage{mathrsfs}
\usepackage[autostyle, english = american]{csquotes}
\usepackage[backend=biber,style=alphabetic,giveninits=true,maxnames=50,autolang=other]{biblatex}
\usepackage{tikz-cd}
\usepackage[shortlabels]{enumitem}
\hypersetup{
pdftitle={The Beilinson--Bloch conjecture for some non-isotrivial varieties over global function fields},
pdfsubject={Mathematics, Algebraic Geometry, Arithmetic Geometry},
pdfauthor={Matt Broe},
pdfkeywords={},
breaklinks=true
}
\newtheoremstyle{style}
{} 
{} 
{\itshape} 
{} 
{\bfseries} 
{.} 
{.5em} 
{} 

\newtheoremstyle{theoremnum}
{} 
{} 
{\itshape} 
{} 
{\bfseries} 
{.} 
{.5em} 
{\thmname{#1}\thmnote{ \bfseries #3}}

\theoremstyle{style}
\newcommand{\comment}[1]{}
\newtheorem{thm}{Theorem}[section]
\newtheorem{cor}[thm]{Corollary}
\newtheorem{prop}[thm]{Proposition}
\newtheorem{lem}[thm]{Lemma}
\newtheorem{conj}[thm]{Conjecture}

\newtheorem*{thm*}{Theorem}
\newtheorem*{cor*}{Corollary}
\newtheorem*{prop*}{Proposition}
\newtheorem*{lem*}{Lemma}
\newtheorem*{conj*}{Conjecture}
\newtheorem*{quest*}{Question}
\newtheorem*{claim*}{Claim}
\newtheorem*{ppty*}{Property}

\theoremstyle{theoremnum}

\theoremstyle{definition}
\newtheorem{defn}[thm]{Definition}

\newtheorem{exmp}[thm]{Example}

\theoremstyle{remark}
\newtheorem{rem}[thm]{Remark}
\newtheorem*{rem*}{Remark}

\makeatletter
\def\eqref{\@ifstar\@eqref\@@eqref}
\def\@eqref#1{\textup{\tagform@{\ref*{#1}}}}
\def\@@eqref#1{\textup{\tagform@{\ref{#1}}}}
\makeatother

\DeclareMathOperator{\Spec}{Spec}
\DeclareMathOperator{\Gal}{Gal}

\DeclareMathOperator{\trdeg}{trdeg}

\DeclareMathOperator{\Aut}{Aut}

\DeclareMathOperator{\Hom}{Hom}

\DeclareMathOperator{\id}{id}

\DeclareMathOperator{\Sing}{Sing}

\DeclareMathOperator{\CH}{CH}

\DeclareMathOperator{\BB}{BB}

\DeclareMathOperator{\Alb}{Alb}
\DeclareMathOperator*{\ord}{ord}
\DeclareMathOperator{\rk}{rk}

\DeclareMathOperator{\Frob}{Frob}

\DeclareMathOperator{\character}{char}
\DeclareMathOperator{\model}{mod}

\DeclareMathOperator{\num}{num}

\DeclareMathOperator{\et}{\acute{e}t}

\DeclareMathOperator{\Gr}{Gr}

\DeclareMathOperator{\cont}{cont}
\DeclareMathOperator{\ar}{ar}
\DeclareMathOperator{\TB}{TB}
\DeclareMathOperator{\Ab}{Ab}

\DeclareMathOperator{\Div}{Div}

\DeclareMathOperator{\spe}{sp}
\DeclareMathOperator{\Ev}{Ev}

\newcommand{\Z}{{\mathbb{Z}}}

\newcommand{\F}{{\mathbb{F}}}

\newcommand{\Q}{{\mathbb{Q}}}

\newcommand{\mf}[1]{\mathfrak{#1}}
\newcommand{\mc}[1]{\mathcal{#1}}

\newcommand{\ol}[1]{\overline{#1}}
\newcommand{\ul}[1]{\underline{#1}}

\numberwithin{equation}{section}

\title{The Beilinson--Bloch conjecture for some non-isotrivial varieties over global function fields}

\author{Matt Broe}

\begin{document}
\linespread{1.3}\selectfont
\begin{abstract}
        The Beilinson--Bloch conjecture is a generalization of the Birch and Swinnerton-Dyer conjecture, which relates the ranks of Chow groups of smooth projective varieties over global fields to the order of vanishing of $L$-functions. We prove the conjecture for certain classes of non-isotrivial varieties over $\F_q(t)$, including some cubic threefolds and fivefolds. We deduce the Birch and Swinnerton-Dyer conjecture for their intermediate Jacobians, and use it to establish new cases of the Tate conjecture over finite fields. We also prove further results on the arithmetic of these intermediate Jacobians. To that end, we show that a few classes of varieties over an arbitrary field have motive of abelian type, generalizing previously known examples over the complex numbers.
\end{abstract}

\maketitle

    \tableofcontents

    \section{Introduction}
    Let $X$ be a smooth projective variety over a global field $k$. The Beilinson--Bloch conjecture (\cite{Bloch1984}, \cite{Beilinson1987}) predicts that 
    \begin{gather}\label{eqBB}
        \dim_{\Q} \CH^i(X)_0 = \ord\limits_{s=i} L(H^{2i-1}(X_{\ol{k}}, \Q_\ell), s).
    \end{gather}
    Here $\CH^i(X)_0$ is the homologically trivial part of the codimension-$i$ Chow group of $X$ (where the Chow group is taken with rational coefficients), and $L(H^{2i-1}(X_{\ol{k}}, \Q_\ell), s)$ is the Hasse--Weil $L$-function associated with the degree-$(2i-1)$ $\ell$-adic cohomology group of $X$. For $X$ an abelian variety and $i = 1$, the equality \eqref{eqBB} is equivalent to the (weak) Birch and Swinnerton-Dyer conjecture for $X$. In general, few cases of the Beilinson--Bloch conjecture have been established unconditionally, particularly over number fields, where the $L$-function is not even known in general to have meromorphic continuation to $s = i$. We recall some existing results on the conjecture in subsection \ref{ssecBB}.
    \par
    In this paper we will be concerned with the Beilinson--Bloch conjecture in positive characteristic. Tate \cite{Tate1965} and Artin--Tate \cite{Tate1966b} recognized in the 1960s that the Birch and Swinnerton-Dyer conjecture over global function fields was closely related to the Tate conjecture for divisors on varieties over finite fields. They conjectured a special value formula for the zeta function of a smooth projective surface over a finite field, and predicted that under mild hypotheses, when the surface was fibered over a curve, the validity of that formula was equivalent to the strong Birch and Swinnerton-Dyer conjecture for the Jacobian of the generic fiber. This equivalence, termed Artin--Tate's Conjecture (d), was proved in the early 2000s by combining results of Artin--Tate, Artin--Grothendieck \cite{Grothendieck1968}, Milne \cite{Milne1986}, and Kato--Trihan \cite{KT2003}. See \cite{LRS2022} for some additional historical background, and two alternative proofs of Conjecture (d). 
    \par
    Geisser in \cite[Theorem 1.1]{Geisser2021} established a higher-dimensional generalization of Conjecture (d), which applies (under mild hypotheses) to any smooth proper variety $X$ over a finite field which is fibered over a curve. Geisser's theorem can be reformulated as saying that the Tate conjecture for divisors on $X$ is equivalent to the Tate conjecture for divisors on the generic fiber, together with the Birch and Swinnerton-Dyer conjecture for the Albanese variety of the generic fiber. Building on work of Jannsen \cite{Jannsen1990}, in \cite[Corollary 4.15]{Broe2025} we generalized part of Geisser's result to cycles of higher codimension, giving a sufficient condition for the Beilinson--Bloch conjecture to hold for a variety over a function field. From this we deduced some new cases of the conjecture for isotrivial varieties (similar isotrivial examples were known previously, such as those given in \cite[Remark 13.6.2]{Jannsen2007} and \cite[Corollary 1]{Kahn2021}). Here a variety over a global function field $K / \F_q$ is called isotrivial if its base change to $\ol{K}$ is defined over $\ol{\F}_q$.
    \par
    It remains to show that this method can prove cases of the Beilinson--Bloch conjecture in higher codimension for interesting non-isotrivial varieties. Since our criterion for the conjecture requires some control over the bad reduction of a variety over a function field, it is natural to seek examples where the singularities of the special fibers are as mild as possible. We are thus led to consider Lefschetz pencils. Recall that a Lefschetz pencil for a smooth projective variety $X$ is roughly a generically smooth family of hyperplane sections of $X$, parametrized by $\mathbb{P}^1$, such that each singular member of the family (over the algebraic closure) is smooth away from one ordinary double point singularity. They have long played a structural role in algebraic and arithmetic geometry, notably in Deligne's papers on the Weil conjectures (\cite{Deligne1974}, \cite{Deligne1980}). 
    \begin{thm}[Theorem \ref{thmBBCubicThreefold}]\label{thmBBCubicThreefoldIntro}
        Let $X$ be a smooth cubic fourfold over $\F_q$, and let $f: Y \to \mathbb{P}^1_{\F_q}$ be a Lefschetz pencil for $X$. Then the Beilinson--Bloch conjecture holds for the generic fiber of $f$.
    \end{thm}
    This establishes the conjecture for a class of smooth cubic threefolds over $\F_q(t)$. Under some stronger hypotheses on $X$, we use results of Jannsen \cite{Jannsen2007} to compute all motivic cohomology groups (with rational coefficients) of the generic fiber of $f$, in terms of certain simpler invariants related to \'etale cohomology (theorem \ref{thmBBM}). We also prove cases of the Beilinson--Bloch conjecture for cubic hypersurfaces of higher dimension (theorem \ref{thmBB3Cubic}), and establish the finiteness of the Albanese kernel for a class of surfaces over $\F_q(t)$ (corollary \ref{corCH2FiniteEx}). The varieties covered by these results are shown to be non-isotrivial in many cases (remark \ref{remNonIsotrivial}): for instance, the cubic threefolds of theorem \ref{thmBBCubicThreefoldIntro} are never isotrivial.
    \par
    The proof of theorem \ref{thmBBCubicThreefoldIntro} relies on the Tate conjecture for cubic fourfolds over finite fields, established unconditionally in \cite{Charles2013}, \cite{Madapusi2015}, and \cite{AKPW2025}. Note that a similar Lefschetz pencil argument was used by Charles--Pirutka \cite{CP2015} in their proof of the integral Tate conjecture for smooth cubic fourfolds in characteristic $\neq 2, 3$.
    \par 
    Smooth cubic threefolds over $\mathbb{C}$ have long been studied through their intermediate Jacobians, which are principally polarized abelian fivefolds constructed from the Hodge structures of the threefolds. Clemens--Griffiths \cite{CG1972} famously took this approach to show that smooth complex cubic threefolds are never rational, a theorem soon extended by Murre \cite{Murre1973} to base fields of characteristic not two. Later work of Murre \cite{Murre1985}, as generalized by Achter--Casalaina-Martin--Vial \cite{ACMV2023}, permits a definition of the intermediate Jacobian of a smooth cubic threefold $X$ over an arbitrary field. This abelian variety satisfies a certain universal property, and admits a concrete geometric interpretation: it is isomorphic to the Albanese variety of the Fano surface of lines on $X$ \cite[Proposition 5.54]{Ciurca2024}. Ciurca in loc. cit. has also announced a proof of the irrationality of smooth cubic threefolds in characteristic two: his manuscript is currently only available as a preprint, and we do not rely on any of its results in the present work.
    \par
    There is a certain relation between the Chow motive of a smooth cubic threefold over any field and that of its intermediate Jacobian (proposition \ref{propH3Cubic}). It formally follows that the Beilinson--Bloch conjecture for a smooth cubic threefold over a global field is equivalent to the Birch and Swinnerton-Dyer conjecture for its intermediate Jacobian. From this and \cite{KT2003}, we deduce the strong form of BSD for the intermediate Jacobians of the cubic threefolds treated in theorem \ref{thmBBCubicThreefoldIntro}. Since the rational $\ell$-adic Tate modules of these abelian varieties are controlled by the well-understood monodromy of Lefschetz pencils, we can make several other observations about their arithmetic. In particular, they have everywhere semistable reduction, and their geometric isogeny classes are simple and non-isotrivial (proposition \ref{propArithA}). We moreover explicitly compute their $L$-functions, and give an example of Mordell--Weil rank 28. In future work, we aim to undertake a more focused study of their BSD invariants. 
    \par
    We also show that a few other classes of smooth complete intersections over any field have associated abelian varieties, well-defined up to isogeny, which serve as motivic analogues of intermediate Jacobians (corollary \ref{corH5Cubic}). In characteristic zero, this follows from combining results of Vial \cite{Vial2013} and Achter--Casalaina-Martin--Vial \cite{ACMV2020}: the proof over a general base field then proceeds via a standard specialization argument. Similarly to the case of cubic threefolds, the Beilinson--Bloch conjecture for a smooth cubic fivefold is equivalent to BSD for the associated \enquote{intermediate Jacobian}. From theorem \ref{thmBB3Cubic}, we thus deduce BSD for the intermediate Jacobians of certain cubic fivefolds over $\F_q(t)$, and give examples of Mordell--Weil rank $\ge 86$.
    \par
    We conclude the paper by proving some new cases of the Tate conjecture over finite fields (theorem \ref{thmTate}). Specifically, we consider a smooth cubic threefold $X$ over a global function field, and assume that $X$ satisfies the Beilinson--Bloch conjecture (which holds unconditionally for the cubic threefolds of theorem \ref{thmBBCubicThreefoldIntro}). Using the Tate conjecture for the Fano surface $F$ of $X$, proved in \cite{Roulleau2013} and \cite[Remark 4.2]{DLR2017}, and the Birch and Swinnerton-Dyer conjecture for the intermediate Jacobian of $X$, combined with the backwards direction of the aforementioned theorem of Geisser, we establish the Tate conjecture for a certain auxiliary threefold over the constant field, obtained by spreading out $F$. 

    \subsection*{Structure of the paper}
    Section 2 recalls some results on the structure of Chow groups of certain varieties, particularly cubic hypersurfaces. Section 3 introduces tools from the theory of Chow motives, and uses them to relate the motive of a smooth cubic threefold with that of its intermediate Jacobian. We prove a similar relation for smooth cubic fivefolds, and a few other complete intersections. The latter proof uses some standard specialization techniques for Chow motives, which nevertheless have further applications which we did not find in the literature. Namely, we use them to generalize to arbitrary base fields several classes of varieties which are known over $\mathbb{C}$ to have finite-dimensional motive. These include the Fano varieties of lines on smooth cubic threefolds and fivefolds.
    \par
    In section 4 we state Jannsen's generalized Tate conjecture and the Beilinson--Bloch conjecture, and give sufficient conditions for them. Section 5 collects results from SGA 7 \cite{SGA7-II} and Weil II \cite{Deligne1980} on the monodromy of Lefschetz pencils. These permit an explicit computation (under mild hypotheses) of the $L$-function associated with the middle cohomology of the generic fiber of a Lefschetz pencil over $\F_q$, in terms of the zeta function of the total space of the pencil. Other than the definitions, the content of this section finds its primary application in section 7.
    \par
    Section 6 proves cases of the Beilinson--Bloch conjecture for generic fibers of certain Lefschetz pencils, essentially via geometric arguments. These are enabled by the simple geometry of the singular fibers of Lefschetz pencils of cubic hypersurfaces (lemma \ref{lemNodalCubic}). Section 7 deduces BSD for the intermediate Jacobian of a smooth cubic threefold over $\F_q(t)$ which arises from a Lefschetz pencil of a cubic fourfold, and uses it to obtain cases of the Tate conjecture over $\F_q$. Here we also employ the earlier results on monodromy of Lefschetz pencils to describe the arithmetic of the intermediate Jacobian, and compute its $L$-function. We prove analogous statements for the intermediate Jacobians of certain smooth cubic fivefolds.

    \subsection*{Notation and conventions}
    The letter $k$ will denote a field (arbitrary unless stated otherwise), and $\ol{k}$ will denote a choice of algebraic closure. We let $G_k = \Aut(\ol{k} / k)$ denote the absolute Galois group of $k$. A variety over $k$ is a finite-type separated integral $k$-scheme. Unless stated otherwise, $\ell$ will always mean a fixed prime not equal to the characteristic of the base field. Our conventions for motives follow Andr\'e's book \cite{Andre2004}, and also match those of the Stacks project \cite[Tag 0FG9]{stacks-project}. By default, all Chow groups and motives have rational coefficients, and are taken with respect to rational equivalence. For a smooth projective variety $X$, the Chow group modulo homological equivalence (with respect to $\ell$-adic cohomology) is denoted by $\CH^i(X)_{\hom}$, and the homologically trivial subgroup of $\CH^i(X)$ is denoted by $\CH^i(X)_0$. By $\CH_i(X)$ we mean the Chow group graded by dimension of cycles, rather than codimension. When $i < 0$, both $\CH^i(X)$ and $\CH_i(X)$ are defined to be zero. The category of motives over $k$ is denoted by $\mc{M}(k)$. The tensor unit in this category is denoted by $\mathbf{1}$. The motive of a smooth projective scheme $Y$ is denoted by $\mf{h}(Y)$. The dual of a motive $M$ is denoted by $M^\vee$, and the dual of a morphism $f: M \to N$ of motives is denoted by $f^\vee: N^\vee \to M^\vee$. The $n$-fold Tate twist of $f$ is denoted by $f(n): M(n) \to N(n)$. Our sign convention for Tate twists of motives is such that $\mf{h}(\mathbb{P}^1) \cong \mathbf{1} \oplus \mathbf{1}(-1)$. Subscripts on schemes, motives, or morphisms thereof denote base change. We will sometimes consider motives and Chow groups modulo other adequate equivalence relations, and also denote this by a subscript (e.g. if $M$ is a motive, then $M_{\num}$ denotes $M$ considered modulo numerical equivalence). If $R$ is a commutative ring, the Chow group with coefficients in $R$ is denoted by $\CH^i(X)_R$, and the category of motives with coefficients in $R$ is denoted by $\mc{M}(k)_R$. By $H^i(Y_{\ol{k}}, \Q_\ell)$ we mean $\ell$-adic \'etale cohomology. For a finite field $\F_q$, $\Frob_q$ denotes the geometric Frobenius in $G_{\F_q}$, which sends $x \mapsto x^{1/q}$. 
    \par
    We caution that some of our references (e.g. \cite{Vial2013} and \cite{Vial2017}) adopt a covariant convention for motives, meaning that they work in the opposite of our category $\mc{M}(k)$, and use the opposite sign convention for Tate twists. We will translate to our preferred conventions when using their results. See \cite[Section 14.1.2]{Kahn_Murre_Pedrini_2007} for a guide to the covariant convention.  

    \subsection*{Acknowledgments}
    I thank Jeff Achter, Niranjan Ramachandran, Shubhankar Sahai, Padma Srinivasan, Doug Ulmer, and Charles Vial for helpful conversations.

    \section{Preliminaries on algebraic cycles}
    We begin by collecting some known results on the Chow groups of certain varieties. Our focus is on cubic hypersurfaces (whose Chow groups are also surveyed in \cite[Section 7.7.4]{Huybrechts2023}). Throughout the paper we will freely use the following standard lemma. 
    \begin{lem}\label{lemChInject}
        For any finite-type scheme $X$ over $k$ and any algebraic field extension $L / k$, the pullback map $\CH_i(X) \to \CH_i(X_L)$ is injective. When $X$ is additionally smooth and separated, the pullback is also injective when $L / k$ is an arbitrary extension.
    \end{lem}
    \begin{proof}
        This follows from \cite[Tag 0FH8]{stacks-project} and \cite[Lemma 2.2]{Broe2025}.
    \end{proof}
    Let $\Omega$ be a universal domain containing $k$, i.e. $\Omega$ is an algebraically closed field of infinite transcendence degree over its prime field. 
    \begin{prop}[{\cite[Lemma 5.1]{FV2023}}]\label{propRatNumCH2}
        Let $X$ be a smooth projective variety over $k$, such that there exists a curve $C \subseteq X_\Omega$ whose points generate $\CH_0(X_\Omega)$. If $H^3(X_{\ol{k}}, \Q_\ell) = 0$, then rational and numerical equivalence agree on $\CH^2(X)$. 
    \end{prop}
    \begin{rem}\label{remCHdim0Cubic}
        In particular, the proposition applies when $X$ is a smooth cubic hypersurface of dimension $\ge 2$ and $\neq 3$: then $X_\Omega$ is unirational \cite[Theorem 1]{Kollar2002}, hence rationally chain connected \cite[Example 10.1.6]{Fulton1998}, so $\CH_0(X_\Omega)_\Z \cong \Z$.
    \end{rem}
    The Fano variety of lines encodes much of the geometry of a cubic hypersurface. We note some of its basic properties.
    \begin{prop}[{\cite[Theorems 1.10 and 1.16]{AK1977}}]
        For $X$ a smooth cubic hypersurface over $k$ of dimension $\ge 3$, the Fano variety of lines on $X$ is smooth, projective, and geometrically connected.
    \end{prop}
    \par
    For a smooth, projective, geometrically connected variety $X$ over $k$, there exists a natural \enquote{Albanese} map from the group $\CH_0(X)_{\Z, 0}$ of zero-cycles of degree zero to the $k$-points of the Albanese variety $\Alb(X)$. We denote by $T(X)$ the kernel of this map, and refer to it as the Albanese kernel of $X$.
    \begin{prop}[{\cite[Proposition 9]{KS1983}}]\label{propAlbKerFin}
        When $k$ is finite, $T(X)$ is finite.
    \end{prop}
    \begin{cor}\label{corCHDim1Cubic}
        If $k$ is an algebraic extension of a finite field, then for any smooth cubic hypersurface $X$ over $k$ of dimension $\ge 3$, we have $\CH_1(X) \cong \Q$.
    \end{cor}
    \begin{proof}
        It suffices to assume that $k = \ol{k}$. Then by \cite[Theorem 1.7]{Shen2019}, there exists a surjective map $\CH_0(F(X)) \to \CH_1(X)$, where $F(X)$ is the Fano variety of lines on $X$. By our assumption on $k$, the group of $k$-points of $\Alb(F(X))$ is torsion, so proposition \ref{propAlbKerFin} implies that $\CH_0(F(X)) \cong \Q$. As any nonzero effective class in $\CH_1(X)$ is nontrivial, we must have $\CH_1(X) \cong \Q$.
    \end{proof}
    We note in passing that \cite[Theorem 1.7]{Shen2019} has been generalized to cycles of higher dimension in \cite[Theorem 0.1]{Lyu2025}, though we will not use this.
    \par
    Work of Esnault--Levine--Viehweg gives some control over Chow groups of general projective schemes in a certain range.

    \begin{prop}[{\cite[Theorem 4.6]{ELV1997}}]\label{propGenChVanish}
        Let $Z$ be a closed subscheme of $\mathbb{P}^n_k$, defined by $r$ homogeneous equations of respective degrees $d_1 \ge...\ge d_r \ge 2$. Fix $j \ge 0$. If either $d_1 \ge 3$ or $r \ge j + 1$, then whenever 
        \begin{gather*}
            \sum\limits_{i=1}^r \binom{j + d_i}{j+1} \le n,
        \end{gather*}
        we have $\CH_{j'}(Z) \cong \Q$ for $0 \le j' \le j$. If $d_1 = 2$ and $1 \le r \le j$, then whenever
        \begin{gather*}
            r(j + 2) \le n + r - j - 1,
        \end{gather*}
        the same conclusion holds. 
    \end{prop}

    For example, in the case of a cubic hypersurface $Z$ of dimension $\ge 5$ over any field, we find that $\CH_1(Z) \cong \Q$.

    \begin{prop}\label{propTateForK3Cubic}
        Let $k$ be a finitely generated field and $X$ a smooth cubic fourfold over $k$. If $\character(k) \neq 2$, or $k$ is finite of characteristic two, then the Tate conjecture holds for $X$.
    \end{prop}
    \begin{proof}
        See \cite[Corollary 6]{Charles2013}, \cite[Theorem 5.14]{Madapusi2015}, and \cite[Theorem 4.9]{AKPW2025}.
    \end{proof}

    When $X$ is a smooth cubic hypersurface of dimension $d$, its middle Betti number $b_d(X)$ is computed in \cite[Corollary 1.1.12]{Huybrechts2023}: when $d = 2,$ $3,$ $4$, $5$, or $6$, $b_d(X)$ is respectively 7, 10, 23, 42, or 87. We will freely use these values throughout the paper without further comment. We also note the classical result that a smooth cubic surface over an algebraically closed field is isomorphic to a blowup of $\mathbb{P}^2$ at six points \cite[Proposition 4.2.4]{Huybrechts2023}.

    \section{Motives}
    In this section, we recall some needed background on motives, and prove a motivic relation between a smooth cubic threefold and its intermediate Jacobian (proposition \ref{propH3Cubic}). An analogous relation holds for a smooth cubic fivefold, and for certain other complete intersections (corollary \ref{corH5Cubic}). The proofs of the latter results rely on standard specialization techniques for Chow motives, which enable the generalization of certain statements about motives over $\mathbb{C}$ to more-or-less arbitrary base fields.
    \par
    The section is organized as follows. In subsections \ref{ssecFdMot}-\ref{ssecBlowup}, we review some known facts about motives, including aspects of finite-dimensional motives and Chow--K\"unneth decompositions. In subsection \ref{ssecCubic3Mot} we give an explicit description of the motive of a smooth cubic threefold, in terms of that of its intermediate Jacobian. In subsection \ref{ssecSpecialization}, we note that Fulton's specialization map on Chow groups \cite[Chapter 20.3]{Fulton1998} can be extended to a specialization functor on motives, and record some properties of this functor. In subsection \ref{ssecRepChow}, we use specialization to generalize results of Vial \cite{Vial2013} and Achter--Casalaina-Martin--Vial \cite{ACMV2020}, on representability of Chow groups and motives over characteristic zero fields, to arbitrary base fields. In subsection \ref{ssecMotAbTypeEx}, we apply this to explicit examples, constructing \enquote{intermediate Jacobians} (up to isogeny) of smooth cubic fivefolds, and of some other smooth complete intersections. Though it is not needed for our immediate applications in this paper, we also generalize to essentially arbitrary base fields other examples of varieties with finite-dimensional motive and motive of abelian type, originally due to Diaz \cite{Diaz2016}, Laterveer (\cite{Laterveer2016}, \cite{Laterveer2017}, \cite{Laterveer2018}, \cite{Laterveer2021}), Bolognesi--Laterveer \cite{BL2024}, and Fu--Moonen \cite{FM2023}. These examples include the Fano varieties of lines on smooth cubic threefolds and fivefolds, Gushel--Mukai threefolds and fivefolds, and a few explicit families of K3 surfaces and cubic fourfolds.
    
    \subsection{Finite-dimensional motives}\label{ssecFdMot} 
    Our later arguments will make ample use of the theory of finite-dimensional motives \cite[Definition 3.7]{Kimura2005}, in which the below result, known as the nilpotence theorem, plays a fundamental role.
    \begin{prop}[{\cite[Proposition 7.5]{Kimura2005}}]\label{propNilp}
        If $M$ is a finite-dimensional motive over $k$, and $f: M \to M$ is an endomorphism which induces the zero map modulo numerical equivalence, then $f$ is nilpotent.
    \end{prop}
    We will repeatedly refer to the fact, due to Vial, that if $M$ is a motive over $k$, and $L / k$ is a field extension, then $M_L$ is finite-dimensional iff $M$ is finite-dimensional \cite[Theorem 2]{Vial2017}.
    \par
    The following standard lemma is quoted from \cite[Section 3.3]{Vial2017}. It implies that two finite-dimensional motives which are isomorphic modulo numerical equivalence are in fact isomorphic.
    \begin{lem}\label{lemFdInverse}
        Let $f: M \to N$ be a map of motives over $k$, and $f_{\num}: M_{\num} \to N_{\num}$ the induced map modulo numerical equivalence. If $N$ is finite-dimensional and $f_{\num}$ has a right inverse, then $f$ has a right inverse. If $M$ is finite-dimensional and $f_{\num}$ has a left inverse, then $f$ has a left inverse.
    \end{lem}
    \begin{proof}
        Suppose that $N$ is finite-dimensional, and let $g_{\num}$ be a right inverse of $f_{\num}$. Let $g: N \to M$ be a lift of $g_{\num}$ modulo rational equivalence. Then by proposition \ref{propNilp}, there exists $n \ge 0$ such that $(\id_N - f \circ g)^n = 0$. By algebra we obtain an explicit $g': N \to M$ with $\id_N = f \circ g'$. The proof of the opposite assertion is identical. 
    \end{proof}

    In view of the previous lemma, it will be useful to have a cohomological criterion for isomorphism of numerical motives.
    \begin{lem}\label{lemNumMotIsoCrit}
        Let $M$, $N$ be motives over $k$, and let $f: M \to N$ and $g: N \to M$ be two maps which induce isomorphisms on $\ell$-adic cohomology. Then $f$ and $g$ induce isomorphisms modulo numerical equivalence.
    \end{lem}
    \begin{proof}
         Let $h = g \circ f$, so that $h$ induces an automorphism $a$ of $V = H^\bullet(M, \Q_\ell)$. Let $P(t) \in \Q_\ell[t]$ be the characteristic polynomial of $a$ acting on $V$. Then we have $P(0) = \det(a) \neq 0$, while $P(a) = 0$ by the Cayley--Hamilton theorem. Multiplying both sides of $P(a) = 0$ by $a^{-1}$ and solving for $a^{-1}$, we find that $a^{-1}$ is a $\Q_\ell$-linear combination of powers of $a$. In particular, $a^{-1}$ is the $\ell$-adic realization of an endomorphism of $M \otimes_\Q \Q_\ell$, considered as a motive with $\Q_\ell$-coefficients. Repeating this argument with $f \circ g$ in place of $g \circ f$, we see that $f \otimes_{\Q}\Q_\ell: M \otimes_\Q \Q_\ell \to N \otimes_\Q \Q_\ell$ induces an isomorphism modulo homological equivalence with respect to $\ell$-adic cohomology. Since numerical equivalence is the coarsest adequate equivalence relation, this implies that the map $f_{\num} \otimes_{\Q} \Q_\ell: M_{\num} \otimes_{\Q} \Q_\ell \to N_{\num} \otimes_{\Q} \Q_\ell$ induced by $f \otimes_{\Q} \Q_\ell$ modulo numerical equivalence is an isomorphism. Motives modulo numerical equivalence form an abelian category \cite{Jannsen1992}, so the kernel and cokernel of $f_{\num}: M_{\num} \to N_{\num}$ must vanish after tensoring with $\Q_\ell$. But this implies that the kernel and cokernel were already zero before tensoring, so $f_{\num}$ is an isomorphism.
    \end{proof}
    Combining lemmas \ref{lemFdInverse} and \ref{lemNumMotIsoCrit} yields the following corollary.
    \begin{cor}\label{corFdIsoCrit}
        With hypotheses and notation as in lemma \ref{lemNumMotIsoCrit}, if $M$ and $N$ are further assumed to be finite-dimensional, then $M \cong N$.
    \end{cor}
    \subsection{Artin motives and motives of abelian type}\label{ssecArtAb} 
    An Artin motive is a motive which is isomorphic to a direct summand of $\mf{h}(\Spec A)$, where $A$ is a finite \'etale $k$-algebra. The functor $M \mapsto \CH^0(M_{\ol{k}})$ is an equivalence from the category of Artin motives over $k$ to the category of continuous $G_k$-representations on finite-dimensional $\Q$-vector spaces \cite[Exemples 4.1.6.1]{Andre2004}. Here the vector spaces are equipped with the discrete topology.
    \par
    A motive of abelian type is an object of the thick and rigid tensor subcategory of $\mc{M}(k)$ generated by the motives of abelian varieties. For instance, the motive of a smooth, projective, geometrically connected curve is of abelian type \cite[Proposition 4.5]{Scholl1994}. When $m \ge 1$ and $\character(k) \nmid m$, the Fermat hypersurface $X \subset \mathbb{P}^{n+1}_k$ defined by the equation
    \begin{gather*}
        \sum\limits_{i=0}^{n+1} x_i^m = 0
    \end{gather*}
    has motive of abelian type \cite[Theorem 1]{KS79}. Artin motives, and motives of abelian type, are finite-dimensional \cite[Example 9.1]{Kimura2005}. 
    \par
    The property of being an Artin motive descends along field extensions.
    \begin{lem}\label{lemArtinBaseChange}
        If $M \in \mc{M}(k)$, and there exists a field extension $L / k$ such that $M_L$ is an Artin motive, then $M$ is an Artin motive.
    \end{lem}
    \begin{proof}
        We may assume that $L$ is algebraically closed, so that $M_L \cong \mathbf{1}_L^{\oplus r}$ for some $r \ge 0$. Then $M$ is finite-dimensional by \cite[Theorem 2]{Vial2017}. Combining lemma \ref{lemFdInverse} with \cite[Proposition 5.5]{Kahn2018} shows that $M_{\ol{k}} \cong \mathbf{1}_{\ol{k}}^{\oplus r}$. The finite-dimensional $G_k$-representation $\CH^0(M_{\ol{k}})$ then has a corresponding Artin motive $N \in \mc{M}(k)$. Since $\CH^0(M_{\ol{k}}) = \Hom_{\mc{M}(\ol{k})}(\mathbf{1}_{\ol{k}}, M_{\ol{k}})$, there is a $G_k$-equivariant isomorphism $f_{\ol{k}}: N_{\ol{k}} \xrightarrow{\sim} M_{\ol{k}}$. By \cite[Lemma 2.2]{Broe2025}, this map descends to an isomorphism $f: N \xrightarrow{\sim} M$ defined over $k$.
    \end{proof}
    \subsection{Chow--K\"unneth decompositions}\label{ssecCK}
    Let $H$ be a classical Weil cohomology theory on $\mc{M}(k)$ \cite[\hphantom{}3.4]{Andre2004}; for instance, $H$ can be \'etale cohomology with coefficients in $\Q_\ell$.
    \begin{defn}
        A \textbf{Chow--K\"unneth decomposition} of a motive $M$  (with respect to $H$) consists of an isomorphism of motives
        \begin{gather*}
            M \cong \bigoplus\limits_{i \in \Z} \mf{h}^i(M),
        \end{gather*}
       such that for all $i, j \in \mathbb{Z}$, we have $H^j(\mf{h}^i(M)) = 0$ whenever $i \neq j$. 
       \par
       If $X$ is a smooth projective scheme over $k$, then we denote the components of a Chow--K\"unneth decomposition of $\mf{h}(X)$ by $\mf{h}^i(X)$. Suppose that $X$ has pure dimension $d$. A Chow--K\"unneth decomposition of $\mf{h}(X)$ is then said to be \textbf{self-dual} if for all $i$, the canonical isomorphism $\mf{h}(X)^\vee \cong \mf{h}(X)(d)$ induces an isomorphism $\mf{h}^i(X)^\vee \cong \mf{h}^{2d-i}(X)(d)$. Equivalently, if $\pi_i: \mf{h}(X) \to \mf{h}^i(X)$ denotes the projector onto the $i$-th Chow--K\"unneth component, the decomposition is self-dual iff for all $i$, $\pi_i$ is the transpose of $\pi_{2d-i}$.
    \end{defn}
    By \cite[Th\'eor\`eme 4.2.5.2 and Remarque 5.1.1.2]{Andre2004}, the above definition does not depend on the choice of $H$, so we will simply speak of Chow--K\"unneth decompositions, with no reference to $H$. Self-dual Chow--K\"unneth decompositions are known to exist for motives of abelian varieties \cite{DM1991}, curves, surfaces \cite{Murre1990}, complete intersections in $\mathbb{P}^n$ \cite{KS2016}, and several other classes of varieties. All the concrete examples of Chow--K\"unneth decompositions considered in this paper will have the property that $\CH^i(\mf{h}^1(X)) = 0$ when $i \neq 1$, and $\CH^1(\mf{h}^1(X)) = \CH^1(X)_0$ (see \cite[Theorem 4.4]{Scholl1994}).
    \par
    For later use, we recall the explicit Chow--K\"unneth decompositions of smooth complete intersections.
    \begin{prop}\label{propFd}
        A smooth complete intersection $X \subset \mathbb{P}^r_k$ of dimension $d$ has a self-dual Chow--K\"unneth decomposition of the form
        \begin{gather*}
            \mf{h}(X) \cong \left[\bigoplus\limits_{\substack{0 \le i \le d \\ i \neq d / 2}} \mathbf{1}(-i)\right] \oplus \mf{h}^d(X).
        \end{gather*}
    \end{prop}
    \begin{proof}
        The existence and form of the Chow--K\"unneth decomposition is proved in \cite[Section 3.5]{KS2016}. 
    \end{proof}

    Up to isomorphism, Chow--K\"unneth components of finite-dimensional motives are independent of the choice of decomposition.

    \begin{lem}\label{lemFdUniqueCK}
        If $M$ is a finite-dimensional motive, then for any two Chow--K\"unneth decompositions $\mf{h}^\bullet(M), \mf{h}^\bullet(M)'$ of $M$, we have $\mf{h}^i(M) \cong \mf{h}^i(M)'$ for all $i$.  
    \end{lem}
    \begin{proof}
        Using the Chow--K\"unneth projectors to construct a map $\mf{h}^i(M) \to \mf{h}^i(M)'$, this follows immediately from corollary \ref{corFdIsoCrit}.
    \end{proof}
    We note that all maps between $\mf{h}^1$s of abelian varieties arise from maps in the isogeny category of abelian varieties.
    \begin{lem}\label{lemAbH1}
        For $A, B$ abelian varieties over $k$, there is a canonical isomorphism
        \begin{gather*}
            \Hom_{\mc{M}(k)}(\mf{h}^1(A), \mf{h}^1(B)) \cong \Hom_k(A^\vee, B^\vee) \otimes_\Z \Q,
        \end{gather*}
        where $A^\vee$ and $B^\vee$ denote the dual abelian varieties of $A$ and $B$, and $\Hom_k(A^\vee, B^\vee)$ denotes homomorphisms of abelian varieties. In particular, if a map $f: \mf{h}^1(A) \to \mf{h}^1(B)$ induces an isomorphism on $\ell$-adic cohomology, then some integer multiple of $f$ is induced by an isogeny from $B$ to $A$, so that $f$ is an isomorphism of motives.
    \end{lem}
    \begin{proof}
        This follows from \cite[Proposition 4.5]{Scholl1994}.
    \end{proof}

    \subsection{The blowup formula}\label{ssecBlowup}
    We will later need the below blowup formula for motives.
    \begin{prop}[{\cite[Corollary 3.5]{Hanamura2011}}]\label{propBlowup}
        If $X$ is a smooth projective variety over $k$, $Z \subseteq X$ is a smooth closed subvariety of codimension $c$, and $\widetilde{X}$ is the blowup of $X$ along $Z$, then there is an isomorphism
        \begin{gather*}
            \mf{h}(\widetilde{X})_\Z \cong \mf{h}(X)_\Z \oplus \bigoplus\limits_{j=1}^{c-1} \mf{h}(Z)(-j)_\Z
        \end{gather*}
        in the category of motives with $\Z$-coefficients over $k$. This induces isomorphisms
        \begin{gather*}
            \CH^i(\widetilde{X})_\Z \cong \CH^i(X)_\Z \oplus \bigoplus\limits_{j=1}^{c-1} \CH^{i - j}(Z)_\Z
        \end{gather*}
        for all $i$.
    \end{prop}
    
    \subsection{Cubic threefolds}\label{ssecCubic3Mot}
    In proposition \ref{propH3Cubic} below, we explicitly describe the motives of smooth cubic threefolds.
    \begin{lem}\label{lemCubicFd}
        When $X$ is a smooth cubic hypersurface of dimension $\le 3$, the motive of $X$ is finite-dimensional.
    \end{lem}
    \begin{proof}
        Curves have finite-dimensional motive, so we may assume $\dim X \ge 2$. When $X$ has a rational point, $X$ is unirational by \cite[Theorem 1]{Kollar2002}, and \cite[Theorem 3.12 and Proposition 3.13]{Vial2017} then imply that $\mf{h}(X)$ is finite-dimensional. Theorem 2 of loc. cit. yields the same conclusion when $X$ does not have a rational point.
    \end{proof}
    The following proposition is likely known to experts, but we did not find a proof over an arbitrary field in the literature. It was proved by Shermenev over an algebraically closed field \cite{Shermenev1970}, via a different method. See \cite[Remark 5.3.17]{Huybrechts2023} for some relevant background and references.
    \begin{prop}\label{propH3Cubic}
        Let $X$ be a smooth cubic threefold over $k$, equipped with a Chow--K\"unneth decomposition. There is then an isomorphism
        \begin{gather*}
            \mf{h}^3(X) \cong \mf{h}^1(A)(-1),
        \end{gather*}
        where $A$ is the Albanese variety of the Fano surface $F$ of lines on $X$.
    \end{prop}
    \begin{proof}
        Fix a self-dual Chow--K\"unneth decomposition 
        \begin{gather*}
            \mf{h}(F) \cong \mathbf{1} \oplus \mf{h}^1(F) \oplus \mf{h}^2(F) \oplus \mf{h}^1(F)(-1) \oplus \mathbf{1}(-2)
        \end{gather*}
        as in \cite[Theorem 4.4]{Scholl1994}, so that $\mf{h}^1(F) \cong \mf{h}^1(A)$. Let $I$ be the incidence correspondence 
        \begin{gather*}
            I = \{(L, x): x \in L\} \subset F \times_k X, 
        \end{gather*}
        and let 
        \begin{gather*}
            p = [I] \in \CH^2(F \times_k X) = \Hom_{\mc{M}(k)}(\mf{h}(F), \mf{h}(X)).
        \end{gather*}
        Then
        \begin{align*}
            p^\vee &\in \Hom_{\mc{M}(k)}(\mf{h}(X)^\vee, \mf{h}(F)^\vee) \\
            &= \Hom_{\mc{M}(k)}(\mf{h}(X)(3), \mf{h}(F)(2)) \\
            &= \Hom_{\mc{M}(k)}(\mf{h}(X), \mf{h}(F)(-1)). 
        \end{align*}
        Let $\pi_i: \mf{h}(X) \to \mf{h}^i(X)$ (resp. $\psi_i: \mf{h}(F) \to \mf{h}^i(F)$) denote the projectors onto the Chow--K\"unneth components. By \cite[page 256]{CP2015}, $p^\vee$ induces an isomorphism $H^3(X_{\ol{k}}, \Z_\ell) \cong H^1(F_{\ol{k}}, \Z_\ell(-1))$, so that 
        \begin{align*}
            f &= \pi_3 \circ \left(p|_{\mf{h}^3(F)}\right): \mf{h}^1(F)(-1) \to \mf{h}^3(X),\\
            g &= \psi_1(-1) \circ \left(p^\vee|_{\mf{h}^3(X)}\right): \mf{h}^3(X) \to \mf{h}^1(F)(-1)
        \end{align*}
        induce isomorphisms on $\ell$-adic cohomology. Then corollary \ref{corFdIsoCrit} yields an isomorphism $\mf{h}^3(X) \cong \mf{h}^1(F)(-1)$. 
    \end{proof}

    \subsection{Specialization}\label{ssecSpecialization}
    An analogue of proposition \ref{propH3Cubic} holds for a smooth cubic fivefold, and certain other complete intersections. Its proof will rely on a specialization argument, which reduces the statement over an arbitrary field to the characteristic zero case.  For this purpose, we now introduce a category $\mc{M}^{\model}(K)$ of \enquote{motives with good models} over $K$, where $K$ is a discretely valued field. 
    \par
    Let $R$ be the valuation ring of $K$, and $\kappa$ the residue field of $R$. The objects of $\mc{M}^{\model}(K)$ are of the form $(\mc{X}, \pi, n)$, where $\mc{X}$ is a smooth projective scheme over $R$, $\pi$ is an idempotent self-correspondence of the generic fiber $\mc{X}_K$, and $n$ is an integer. In other words, $(\mc{X}_K, \pi, n)$ is a Chow motive over $K$ in the usual sense. Morphisms in $\mc{M}^{\model}(K)$ are simply morphisms of the underlying Chow motives over $K$. The category $\mc{M}^{\model}(K)$ naturally inherits the structure of a rigid, pseudo-abelian tensor category from $\mc{M}(K)$, since the usual constructions of direct sums and summands, tensor products, and duals in $\mc{M}(K)$ all extend in an evident way to account for the extra data of a smooth projective scheme over $R$. An object of $\mc{M}^{\model}(K)$ is finite-dimensional in the sense of \cite[D\'efinitions 3.3]{Andre2005} iff the underlying motive over $K$ is finite-dimensional. 
    \par
    Note that $\mc{M}^{\model}(K)$ is slightly different from the category of relative Chow motives over $R$, since the morphisms in the former are correspondences defined over $K$, rather than over $R$. See \cite{DM1991}, \cite{CH2000}, \cite{CH2007}, \cite{Moonen2014}, and \cite{Guletskii2018} for some previous works which consider relative Chow motives, and also \cite{CD2019} for the general formalism of triangulated categories of mixed motives over a base. 
    \par
    For a smooth projective scheme $\mc{X}$ over $R$, Fulton in \cite[Chapter 20.3]{Fulton1998} constructs a specialization homomorphism $\sigma: \CH^\bullet(\mc{X}_K)_\Z \to \CH^\bullet(\mc{X}_\kappa)_\Z$ satisfying the following compatibilities:
    \begin{enumerate}[(i)]
        \item $\sigma$ is a map of graded rings;
        \item $\sigma$ is compatible with the $\ell$-adic cycle class map\footnote{Fulton only states the compatibility of specialization with the $\ell$-adic cycle class map under the assumption that $R$ is complete, but the general case follows by base change to the completion of $R$. For a closed subvariety $Z \subseteq \mc{X}_K$, $\sigma([Z])$ is simply the class of the special fiber of the closure of $Z$ in $\mc{X}$, so $\sigma$ is evidently compatible with this base change. In any case, nothing is lost in the following subsections by always requiring $R$ to be complete.}, via the smooth proper base change isomorphism $H^{2i}(\mc{X}_{\ol{K}}, \Z_\ell(i)) \cong H^{2i}(\mc{X}_{\ol{\kappa}}, \Z_\ell(i))$;
        \item for a map $f: \mc{Y} \to \mc{X}$ of smooth projective schemes over $R$, $\sigma$ commutes with pushforward of cycles by $f$, i.e. for a cycle $c \in \CH^i(\mc{Y}_K)$, we have $\sigma(f_{K, *}(c)) = f_{\kappa, *}(\sigma(c))$;
        \item when $f$ is flat, $\sigma$ commutes with pullback of cycles by $f$. 
    \end{enumerate}
    The specialization map is thus compatible with the composition of correspondences of smooth projective varieties, and induces a motivic specialization functor $\sigma_{\mc{M}}: \mc{M}^{\model}(K) \to \mc{M}(\kappa)$, sending $(\mc{X}, \pi, n) \mapsto (\mc{X}_\kappa, \sigma(\pi), n)$. This is a functor of $\Q$-tensor categories (in the sense of \cite{Andre2005}), and hence preserves finite-dimensional objects. By smooth proper base change, the specialization functor also preserves Chow--K\"unneth decompositions. Note that motivic specialization techniques based on \cite[Chapters 6 and 20.3]{Fulton1998} are standard; see e.g. \cite[Proof of Theorem 1]{Vial2017}. 
    \par
    When $\mc{X}$ is an abelian scheme over $R$, Deninger--Murre's canonical Chow--K\"unneth decompositions of $\mf{h}(\mc{X}_K)$ and $\mf{h}(\mc{X}_\kappa)$ are compatible with specialization, by the uniqueness statement of \cite[Theorem 5.2]{Scholl1994}. Likewise, for a relative smooth complete intersection $\mc{X}$ over $R$, the Chow--K\"unneth decompositions of $\mf{h}(\mc{X}_K)$ and $\mf{h}(\mc{X}_\kappa)$ discussed in proposition \ref{propFd} can be chosen compatibly with specialization, via the construction of \cite[Section 3.5]{KS2016}.
    \par
    Using specialization, we can prove that the base change functor on motives is conservative, i.e. detects isomorphisms.
     \begin{lem}\label{lemMotBaseChgConser}
        If $f: M \to N$ is a map of motives over $k$, $L / k$ is a field extension, and the base change $f_L: M_L \to N_L$ of $f$ to $L$ is an isomorphism, then $f$ is an isomorphism.
    \end{lem}
    \begin{proof}
        When $L / k$ is algebraic, this follows from a standard Galois descent argument, cf. \cite[Lemma 2.2]{Broe2025}. We may thus assume that $k$ is algebraically closed. By spreading out, it suffices to assume that $L$ is finitely generated over $k$. By induction on the transcendence degree, we can assume that $\trdeg(L / k) = 1$. Finally, since we know the claim for algebraic extensions, we can assume that $L = k(t)$. Then $L$ is the function field of $\mathbb{A}^1_k$, and we can specialize $(f_L)^{-1}$ to a $k$-point of $\mathbb{A}^1_k$, obtaining a map $g: N \to M$ defined over $k$. Since $\sigma_{\mc{M}}$ is a functor, $g$ is the desired inverse of $f$.
    \end{proof}
    \subsection{Representability of Chow groups and motives}\label{ssecRepChow}
    Subject to some conditions, our next theorem \ref{thmCKSpecialize} relates the motive of a smooth projective variety over an arbitrary field to the motives of certain abelian varieties. In characteristic zero, the theorem follows from combining results of \cite{Vial2013} and \cite{ACMV2020}; we then deduce the positive characteristic case via a specialization argument. 
    \par
    In order to state the theorem at a natural level of generality, we must make a few definitions. For our main applications, only a more concrete corollary of the theorem (corollary \ref{corH5Cubic}), describing the motives of some explicit complete intersections, will ultimately be relevant. The following two definitions are standard and will mostly be treated as a black box.
    \begin{defn}\label{defnRep}
        For a smooth projective variety $X$ over a universal domain $\Omega$, let $A^i(X) \subseteq \CH^i(X)$ denote the subgroup of cycles which are trivial modulo algebraic equivalence. We say that $A^i(X)$ is \textbf{representable} if there exists a smooth projective curve $C$ over $\Omega$, and a correspondence $\Gamma \in \CH^i(X \times_\Omega C)$, such that $\Gamma_* A^1(C) = A^i(X)$.
    \end{defn}
    \begin{defn}
        For a smooth projective variety $X$ over $k$, and $i \in \Z$, the \textbf{geometric coniveau filtration} $N^\bullet H^i(X_{\ol{k}}, \Q_\ell)$ on the $\ell$-adic cohomology of $X$ is defined via 
        \begin{gather*}
            N^j H^i(X_{\ol{k}}, \Q_\ell) = \sum\limits_{\substack{Z \subseteq X \textrm{ closed} \\
            \textrm{of codim } \ge j}} \ker(H^i(X_{\ol{k}}, \Q_\ell) \to H^i((X-Z)_{\ol{k}}, \Q_\ell)).
        \end{gather*}
    \end{defn}
    We note that the geometric coniveau filtration is stable under base change along field extensions. 
    \begin{lem}\label{lemConiveauBaseChange}
        For a smooth projective variety $X$ over $k$, a field extension $L / k$, a compatible extension of algebraic closures $\ol{L} / \ol{k}$, and $i, j \in \mathbb{Z}$, pullback induces an isomorphism
        \begin{gather*}
            N^j H^i(X_{\ol{k}}, \Q_\ell) \cong N^j H^i(X_{\ol{L}}, \Q_\ell)
        \end{gather*}
    \end{lem}
    \begin{proof}[Sketch of proof]
        This follows from a standard spreading-out and specialization argument (see the start of \cite[Section 1.2]{ACMV2017}).
    \end{proof}
    The next (somewhat ad hoc) definition organizes the inputs to later specialization arguments. 
    \begin{defn}
        For a smooth projective variety $X$ over $k$, a \textbf{lifting datum} for $X$ is a tuple $(X, R, L, \iota, \mc{X}, \mc{X}_L)$, where 
        \begin{enumerate}[(i)]
            \item $R$ is a discrete valuation ring with residue field $k$, and fraction field $K$ of characteristic zero;
            \item $L$ is a subfield of $K$ equipped with an embedding $\iota: L \to \mathbb{C}$;
            \item $\mc{X}$ is a smooth projective scheme over $R$ with special fiber isomorphic to $X$;
            \item $\mc{X}_L$ is a smooth projective scheme over $L$ whose base change to $K$ is isomorphic to $\mc{X}_K$.
        \end{enumerate}
        We denote by $\mc{X}_{\mathbb{C}}$ the base change of $\mc{X}_L$ along $\iota$.
        \par
        If $k$ has characteristic zero, $R = k[[t]]$, $L = k$, $\mc{X} = X \times_k R$, and $\mc{X}_L = X$, then we say that the lifting datum is \textbf{constant}. In this case, a lifting datum is really just the datum of an embedding $k \hookrightarrow \mathbb{C}$.
    \end{defn}
    We can now state the theorem, which will imply that certain smooth complete intersections over any field have motive of abelian type. See also the related \cite[Theorem 3.4]{Vial2010}, which applies over an algebraically closed base field.
    \begin{thm}\label{thmCKSpecialize}
        Let $X$ be a smooth, projective, geometrically connected variety over $k$ of dimension $d$, and let $(X, R, L, \iota, \mc{X}, \mc{X}_L)$ be a lifting datum for $X$. Then 
        \begin{enumerate}[(i)]
            \item for all $0 \le n \le d-1$, there exists an abelian variety $J^{2n + 1}_a$ over $k$, and a map of motives $\gamma^{2n+1}: \mf{h}^1(J^{2n+1}_a)(-n) \to \mf{h}(X)$, with the following properties:
        \begin{enumerate}[(a)]
            \item for all primes $\ell$ invertible in $k$, the map $\gamma^{2n+1}_\ell$ induced by $\gamma^{2n+1}$ on $\ell$-adic cohomology is a split inclusion of $G_k$-representations;
            \item after identifying $H^\bullet(X_{\ol{k}}, \Q_\ell) \cong H^\bullet(\mc{X}_{\ol{K}}, \Q_\ell)$ via smooth proper base change, the image of $\gamma^{2n+1}_\ell$ is the $n$-th step of the geometric coniveau filtration on $H^{2n+1}(\mc{X}_{\ol{K}}, \Q_\ell)$.
        \end{enumerate}
            \item If the lifting datum is constant, then there exists a distinguished choice of an abelian variety $J^{2n+1}_a$ satisfying (i).
            \item If $\mc{X}_L$ has a Chow--K\"unneth decomposition, and for all $i$, $A^i(\mc{X}_{\mathbb{C}})$ is representable, then $X$ has a Chow--K\"unneth decomposition of the form
        \begin{gather}\label{eqRepCKDecomp}
            \mf{h}(X) \cong \left[\bigoplus\limits_{0 \le i \le d} M^{2i}(-i) \right] \oplus \left[\bigoplus\limits_{0 \le n \le d-1} \mf{h}^1(J^{2n+1}_a)(-n) \right],
        \end{gather}
        where $J_a^{2n+1}$ is as in (i), and the $M^{2i}$ are Artin motives. If the lifting datum is constant, then here we may take $J_a^{2n+1}$ as in (ii).
        \end{enumerate}
    \end{thm}
    We remark that if $\mf{h}(X)$ has a Chow--K\"unneth decomposition of the form \eqref{eqRepCKDecomp}, then lemma \ref{lemFdUniqueCK} implies that any other Chow--K\"unneth decomposition of $\mf{h}(X)$ is of the same form (up to isomorphisms of the summands).
    \begin{proof}[Proof of theorem \ref{thmCKSpecialize}]
        If $L = k = \mathbb{C}$, $\iota = \id$, and the lifting datum is constant, then we may take $J^{2n+1}_a$ to be the $n$-th algebraic intermediate Jacobian of $X$. The claims are then implied by \cite[Theorem A]{ACMV2020} and \cite[Theorem 4]{Vial2013}.
        \par
        Suppose now that $k$ has characteristic zero, and that the lifting datum is constant, so $L = k$. Then \cite[Theorem A]{ACMV2020} and lemma \ref{lemConiveauBaseChange} imply claims (i) and (ii). For (iii), fix a Chow--K\"unneth decomposition for $X$, and assume that for all $i$, $A^i(X_\mathbb{C})$ is representable. Then $\mf{h}(X)$ is finite-dimensional, since this can be checked after base change to $\mathbb{C}$, using the previous paragraph and \cite[Theorem 2]{Vial2017}. Theorem A of \cite{ACMV2020} provides a distinguished model $J_{a}^{2n+1}$ over $k$ of the $n$-th algebraic intermediate Jacobian of $X_{\mathbb{C}}$, together with a map $\gamma^{2n+1}: \mf{h}^1(J^{2n+1}_a)(-n) \to \mf{h}(X)$ inducing a split injection on $\ell$-adic cohomology. We claim that the
        composite
        \begin{gather*}
            g: \mf{h}^1(J^{2n+1}_a)(-n) \xrightarrow{\gamma^{2n+1}} \mf{h}(X) \twoheadrightarrow \mf{h}^{2n+1}(X)
        \end{gather*}
        is an isomorphism. To see this, note that after base change to $\mathbb{C}$, we have $\mf{h}^1(J_a^{2n+1})(-n)_{\mathbb{C}} \cong \mf{h}^{2n+1}(X)_{\mathbb{C}}$, by the previous paragraph and lemma \ref{lemFdUniqueCK}. Then the base change $g_{\mathbb{C}}: \mf{h}^1(J_a^{2n+1})(-n)_{\mathbb{C}} \to \mf{h}^{2n+1}(X)_{\mathbb{C}}$ of $g$ to $\mathbb{C}$ is an isomorphism, since lemma \ref{lemAbH1} allows us to check this after passing to $\ell$-adic cohomology. Now lemma \ref{lemMotBaseChgConser} implies that $g$ is an isomorphism. It remains to examine the even-degree Chow--K\"unneth summands of $X$. By the previous paragraph and lemma \ref{lemFdUniqueCK}, $\mf{h}^{2i}(X)(i)_{\mathbb{C}}$ is isomorphic to a direct sum of copies of the unit motive. We conclude from lemma \ref{lemArtinBaseChange} that $\mf{h}^{2i}(X)(i)$ is isomorphic to an Artin motive, which establishes (iii) in this case.
        \par
        Finally, suppose that $k$ is arbitrary and the lifting datum is not constant. By applying the previous case of the proof to the constant lifting datum for $\mc{X}_L$ associated with $\iota$, and base changing from $L$ to $K$, we obtain an abelian variety $J^{2n+1}_{a, K}$ over $K$ and a correspondence $\gamma^{2n+1}_{K}: \mf{h}^1(J^{2n+1}_{a, K})(-n) \to \mf{h}(\mc{X}_K)$ inducing a split injection on $\ell$-adic cohomology. As $\mc{X}_K$ has good reduction, the N\'eron--Ogg--Shafarevich criterion \cite[Theorem 7.4.5]{BLR1990} and smooth proper base change imply that $J^{2n+1}_{a, K}$ has good reduction to an abelian variety $J^{2n+1}_a$ over $k$. Similarly, if the hypotheses of (iii) are satisfied, then we obtain Artin motives $M^{2i}_K$ over $K$ and a Chow--K\"unneth decomposition 
        \begin{gather*}
            \mf{h}(\mc{X}_K) \cong \left[\bigoplus\limits_{0 \le i \le d} M^{2i}_K(-i) \right] \oplus \left[\bigoplus\limits_{0 \le n \le d-1} \mf{h}^1(J^{2n+1}_{a, K})(-n) \right].
        \end{gather*}
        Smooth proper base change implies that $H^\bullet(M^{2i}_K, \Q_\ell)$ is an unramified $G_K$-representation, which we can thus consider as a $G_k$-representation. By the correspondence between Artin motives and $G_K$-representations, it follows that $M^{2i}_K$ is a direct summand of $\mf{h}(\Spec A)$, where $A = \mc{A} \otimes_R K$ for some finite \'etale $R$-algebra $\mc{A}$. Writing $M^{2i}_K \cong (\Spec A, \pi, 0)$, we find that $M^{2i}_K$ is the underlying Chow motive of $(\Spec \mc{A}, \pi, 0) \in \mc{M}^{\model}(K)$. Applying the motivic specialization functor now reduces claims (i) and (iii) to the previously treated cases.
    \end{proof}
    \begin{rem}
        Via related specialization arguments, one can obtain corollaries of the other main theorems of \cite{Vial2013} over positive characteristic fields.
    \end{rem}
    \subsection{Some varieties with motive of abelian type}\label{ssecMotAbTypeEx}
    It follows from theorem \ref{thmCKSpecialize} that several explicit classes of varieties have motive of abelian type. We will use the following flatness criterion to lift these varieties to characteristic zero.
    \begin{lem}\label{lemFlatnessCrit}
        Let $R$ be a one-dimensional, local, Noetherian integral domain, with maximal ideal $\mf{m}$ and residue field $\kappa$. Let $S = \Spec R$, and let $f: \mc{Y} \to S$ be a finite-type, flat morphism, with special fiber $Y$. Let $\mc{Z}_1,...,\mc{Z}_n$ be closed subschemes of $\mc{Y}$, defined by locally principal ideal sheaves, and let $\mc{Z}$ be their scheme-theoretic intersection, with special fiber $Z$. Assume that $Y$ is Cohen--Macaulay, and that for all points $z \in Z$, the codimension of $Z$ in $Y$ at $z$ is equal to $n$. Then $\mc{Z}$ is flat over $S$. If we further assume that $f$ is proper, and $Z$ is smooth over $\kappa$, then $\mc{Z}$ is smooth over $S$.
    \end{lem}
    \begin{proof}
        First, we will show that $\mc{Z}$ is flat over $S$. Let $z \in Z$. Since the generic fiber of $\mc{Z}$ is automatically flat over $S$, it will suffice to show that $\mc{Z} \to S$ is flat at $z$. Let $B = \mc{O}_{\mc{Y}, z}$ be the local ring of $\mc{Y}$ at $z$, and let $f_i \in B$ be a local generator for the ideal sheaf on $\mc{Y}$ defining $\mc{Z}_i$. Let $C = B / \mf{m} B$, let $\ol{f}_i \in C$ be the reduction modulo $\mf{m}$ of $f_i$, and let $I = \langle \ol{f}_1,...,\ol{f}_n\rangle \subseteq C$. We assumed that the codimension of $Z$ in $Y$ at $z$ was equal to $n$, and that $C$ was Cohen--Macaulay. From \cite[Tag 00N6]{stacks-project} we deduce that $\ol{f}_1,...,\ol{f}_n$ is a regular sequence in $C$. Then \cite[Tag 0470]{stacks-project} implies that $\mc{Z} \to S$ is flat at $z$, as desired.
        \par
        Suppose now that $\mc{Y}$ is proper over $S$, and $Z$ is smooth over $\kappa$. Then for all $z \in Z$, the smooth locus $U$ of $\mc{Z} \to S$ contains $z$, by \cite[Tag 01V9]{stacks-project}. Since $\mc{Z}$ is proper over $S$, $f(\mc{Z} - U)$ is a closed subset of $S$ which does not contain $\mf{m}$, and must therefore be empty. Hence $\mc{Z}-U$ is empty, and $\mc{Z}$ is smooth over $S$.  
    \end{proof}
    \begin{cor}\label{corH5Cubic}
        Let $X$ be a smooth complete intersection in $\mathbb{P}^r_k$, of dimension $d = 2n + 1$ and multidegree $\ul{a}$. Assume that $(d, \ul{a}) \in \{(5, (3)), (3, (3, 2)), (7, (2, 2))\}$. Then the motive of $X$ is of abelian type. In particular, for any Chow--K\"unneth decomposition of $X$, there is an isomorphism
        \begin{gather*}
            \mf{h}^d(X) \cong \mf{h}^1(A)(-n),
        \end{gather*}
        where $A$ is some abelian variety over $k$. If $k$ is equipped with an embedding into $\mathbb{C}$, then there exists a distinguished choice for such an $A$.
    \end{cor}
    \begin{proof}
        There exists a discrete valuation ring $R$ with residue field $k$, and fraction field $K$ of characteristic zero \cite[Theorem 29.1]{Matsumura1989}. Taking $\mc{Y} = \mathbb{P}^r_R$ in lemma \ref{lemFlatnessCrit}, we see that if we arbitrarily lift homogeneous equations for $X$ to $R$, we obtain a relative smooth complete intersection $\mc{X}$ over $R$ with special fiber $X$. Now by spreading out $\mc{X}_K$ over a finitely generated subfield $L \subseteq K$, and choosing an embedding $L \to \mathbb{C}$, we obtain a lifting datum for $X$. 
        \par
        Note that by proposition \ref{propFd}, any smooth complete intersection has a Chow--K\"unneth decomposition, and by lemma \ref{lemFdUniqueCK}, a Chow--K\"unneth decomposition of a finite-dimensional motive is unique up to isomorphism. By proposition \ref{propGenChVanish} and our assumption on $(d, \ul{a})$, we have $\CH_i(\mc{X}_\mathbb{C}) \cong \Q$ for $0 \le i \le \lfloor \frac{d}{2} \rfloor - 1$. Then \cite[Remark 3.8]{Vial2013} implies that $A^i(\mc{X}_\mathbb{C})$ is representable for all $i$. Appealing to theorem \ref{thmCKSpecialize} now completes the proof.
    \end{proof}
    \begin{rem}\label{remCubicFivefold}
        Let $X$ be a general smooth cubic fivefold over a subfield $k$ of $\mathbb{C}$, and $F$ the Fano surface of planes on $X$. Collino in \cite{Collino1986} proved that the incidence correspondence induces an isomorphism from $\Alb(F_{\mathbb{C}})$ to the intermediate Jacobian of $X_{\mathbb{C}}$. By \cite[Proposition 5.1]{ACMV2020}, this map descends to $k$, and gives an isomorphism from $\Alb(F)$ to the distinguished abelian variety $A$ that corollary \ref{corH5Cubic} associates with $X$. 
        \par
        Suppose instead that $X$ is a smooth cubic fivefold over a field $k$ of positive characteristic, that the Fano surface $F$ of planes on $X$ is smooth, and that $X$ lifts to a cubic fivefold in characteristic zero which is general in the sense of Collino's result. By specialization, here as well the incidence correspondence induces an isomorphism $\mf{h}^5(X) \cong \mf{h}^1(\Alb(F))(-2)$. See the proof of theorem \ref{thmFanoAbType} for a more detailed argument along these lines.
    \end{rem}
    \begin{rem}
        Let $X$ be a smooth proper scheme over $k$. In \cite{ACMV2023}, the notion of a regular homomorphism from families of algebraically trivial codimension-$i$ cycles on $X$ to an abelian variety is defined, generalizing the classical definition in \cite{Murre1985}. A codimension-$i$ algebraic representative for $X$ is then defined to be an initial abelian variety $\Ab^i_{X / k}$ receiving such a regular homomorphism. When $i \in \{1,2,\dim X\}$, Theorem 2 of \cite{ACMV2023} establishes the existence of such an algebraic representative for $X$. For $X$ a smooth complete intersection as in corollary \ref{corH5Cubic}, it seems possible that a codimension-$\left(\frac{\dim X+1}{2}\right)$ algebraic representative exists for $X$, and that this algebraic representative is isogenous to the abelian variety $A$ of the corollary. We do not develop this further here, as it is not needed for our immediate applications. 
    \end{rem}
    Next, we show via a motivic specialization argument that the Fano surface of a smooth cubic threefold over any field has motive of abelian type. Over fields of characteristic not two, the below result is due to Diaz \cite{Diaz2016}. 
    \begin{thm}\label{thmFanoAbType}
        Let $X$ be a smooth cubic threefold over $k$, and $F$ the Fano surface of $X$. Then any Chow--K\"unneth decomposition of $F$ satisfies $\mf{h}^2(F) \cong \mf{h}^2(\Alb(F))$. 
    \end{thm}
    \begin{proof}
        Fix a Chow--K\"unneth decomposition 
        \begin{gather}\label{eqFanoCK}
            \mf{h}(F) \cong \mathbf{1} \oplus \mf{h}^1(F) \oplus \mf{h}^2(F) \oplus \mf{h}^1(F)(-1) \oplus \mathbf{1}(-2)
        \end{gather}
        as in \cite[Theorem 4.4]{Scholl1994}, so that $\mf{h}^1(F) \cong \mf{h}^1(\Alb(F))$. When $k$ has characteristic not two, \cite[Theorem 1.2]{Diaz2016} shows that $\mf{h}^2(F) \cong \mf{h}^2(\Alb(F))$. In particular, the motive of $F$ is finite-dimensional in this case.
        \par
        Now suppose that $k$ is arbitrary. Then lemma \ref{lemFlatnessCrit} implies that $X$ lifts to a relative smooth cubic threefold over a DVR $R$ with residue field $k$ and characteristic zero fraction field $K$. This induces a lift of $F$ to a relative Fano surface $\mc{F}$, which is smooth over $R$ by \cite[Theorem 4.2]{AK1977}. We fix a Chow--K\"unneth decomposition of the form \eqref{eqFanoCK} for the generic fiber $\mc{F}_K$ of $\mc{F}$. Since $\mc{F}_K$ has finite-dimensional motive by Diaz's result, its specialization $F$ does as well. Lemma \ref{lemFdUniqueCK} then implies that all Chow--K\"unneth decompositions of $\mf{h}(F)$ are isomorphic, so in particular, the specialization of $\mf{h}^i(\mc{F}_K)$ is isomorphic to $\mf{h}^i(F)$.
        \par
        By the N\'eron--Ogg--Shafarevich criterion, $\Alb(\mc{F}_K)$ has good reduction. Since its reduction $A$ satisfies $\mf{h}^1(A) \cong \mf{h}^1(\Alb(F))$, lemma \ref{lemAbH1} shows that $A$ is isogenous to $\Alb(F)$. We thus have $\mf{h}^2(\Alb(F)) \cong \mf{h}^2(A)$, which by specialization is isomorphic to $\mf{h}^2(F)$, as desired.
    \end{proof}
    The Tate conjecture for the Fano surface of a smooth cubic threefold over a finitely generated field was originally proved in \cite{Roulleau2013} and \cite[Remark 4.2]{DLR2017}. It is also a formal consequence of the above theorem and the Tate conjecture for divisors on abelian varieties (established in \cite{Tate1966}, \cite{Faltings1986}, \cite{Zarhin1974a}, \cite{Zarhin1974b}, and in Mori's results in \cite{MB1985}).
    \par
   \begin{rem}\label{remFdList}
        Remark 2.2 of \cite{BL2024} provides a partial list of varieties over $\mathbb{C}$ known to have finite-dimensional motive. It seems likely that motivic specialization arguments (as in the proof of corollary \ref{corH5Cubic}), coupled with \cite[Theorem 2]{Vial2017}, can extend many such examples to essentially arbitrary base fields. For instance, the statement of \cite[Corollary 17]{Laterveer2017} immediately extends via this method to arbitrary base fields (though note that one must take some care in characteristic three, since there the explicit equations listed in loc. cit. may not define smooth varieties). In particular, the Fano variety of lines on a smooth cubic fivefold over any field has finite-dimensional motive. The main theorem of \cite{Laterveer2016} similarly yields an explicit family of K3 surfaces with finite-dimensional motive over any field of characteristic not two, and the main theorem of \cite{BL2024} yields another such family over any field of characteristic not three. Next, from \cite{Laterveer2018} we obtain a family of smooth cubic fourfolds with finite-dimensional motive over any field of characteristic not three. We will later make use of this example in corollary \ref{corTBForCubicFour} and theorem \ref{thmBBM}. 
        \par
        Finally, let $X$ be a Gushel--Mukai variety \cite[Definition 2.1]{FM2025} over a field $k$ of characteristic not two, with $\dim X \in \{3, 5\}$. Then $X$ has finite-dimensional motive. To see this, we can assume that $k$ is algebraically closed, and use lemma \ref{lemFlatnessCrit} to lift $X$ to a Gushel--Mukai variety in characteristic zero (see also \cite[Lemma 2.2]{FM2025}). The finite-dimensionality then follows from \cite[Proposition 7.1]{FM2023} and specialization. 
   \end{rem}

    \section{Conjectures on algebraic cycles}
    We now review the conjectures on algebraic cycles we will be concerned with, and give some sufficient conditions for them. Proposition \ref{propBBCrit} shows that a sufficiently strong form of the Tate conjecture over finite fields (conjecture \ref{conjGenBeilinson}) implies the Beilinson--Bloch conjecture over function fields, so our main results will hinge on verifying cases of the former conjecture.

    \subsection{The generalized conjectures of Tate and Beilinson}\label{ssecTB}
    For $k$ a field, $X$ a scheme with finite-type, separated structure morphism $f: X \to \Spec k$, and $i, j \in \Z$, we define the $\ell$-adic \'etale homology of $X$ as in \cite[Example 6.7]{Jannsen1990}, via
    \begin{gather*}
        H_i(X_{\ol{k}}, \Q_\ell(j)) = H^{-i}(X_{\ol{k}}, (R{\ol{f}}^! \Q_\ell)(-j)).
    \end{gather*}
    Here $\ol{f}$ is the base change of $f$ to $\ol{k}$.
    \par
    Define the (rational) motivic homology and cohomology of $X$ as in \cite[5.12 and Example 6.12]{Jannsen1990}, via
    \begin{align*}
        H_i^{\mc{M}}(X, \Q(j)) &= K'_{i-2j}(X)^{(-j)},\\
        H^i_{\mc{M}}(X, \Q(j)) &= K_{2j-i}(X)^{(j)}.
    \end{align*}
    Here $K_m(X)$ denotes the $m$-th Quillen $K$-theory of vector bundles on $X$, and $K_{m}(X)^{(r)}$ denotes the simultaneous eigenspace of $K_{m}(X) \otimes_\Z \Q$ for the Adams operations $\{\psi^n: n \in \Z - \{0\}\}$, on which $\psi^n$ acts as multiplication by $n^r$. The definition of $K'_{m}(X)^{(r)}$ is the same, except using the $K$-theory of coherent sheaves on $X$ instead of vector bundles. See \cite[\hphantom{}1.5 and 7.2]{Soule1985} for the definition of the Adams operations. 
    \par
    For context, we note that for a smooth variety $X$ over $k$, and $i$, $j \in \Z$, there is a canonical isomorphism
    \begin{gather*}
        H^i_{\mc{M}}(X, \Q(j)) \cong \CH^j(X, 2j-i)_\Q,
    \end{gather*}
    cf. \cite[page 269(vii)]{Bloch1986} and \cite[Example 6.12]{Jannsen1990}. On the right is Bloch's higher Chow group (with $\Q$-coefficients). For such an $X$, we thus let $H^i_{\mc{M}}(X, \Z(j))$ denote $\CH^j(X, 2j-i)_\Z$. The more modern definitions of integral motivic cohomology and Borel--Moore homology are not in terms of higher Chow groups or $K$-theory \cite[Definitions 3.4 and 16.20]{MVW2006}, but this comparatively classical approach will suffice for us. Corollary 5.3 of \cite{Levine1994} shows that $\Q$-linear correspondences of smooth projective varieties act on higher Chow groups with $\Q$-coefficients, so $X \mapsto H^i_{\mc{M}}(X, \Q(j))$ extends to a covariant functor on $\mc{M}(k)$.
    \par
    In any case, we will not be too preoccupied with the precise definition of motivic (co)homology. Both \'etale and motivic (co)homology form twisted Poincar\'e duality theories in the sense of \cite[Definition 6.1]{Jannsen1990}\footnote{There is a caveat in the case of motivic (co)homology, which forms an \enquote{absolute} twisted Poincar\'e duality theory, rather than a relative one. See \cite[Example 6.12]{Jannsen1990}.}, so that we have the following purely formal tools for working with them. Out of an abundance of caution, in the remainder of this section we work in the category of quasi-projective schemes over the base field, unless otherwise noted\footnote{We could possibly get away with weaker hypotheses, but \cite{Soule1985} as stated works in this generality.}.
    \begin{prop}\label{propHomologyProps}
        For $X$ a quasi-projective scheme over $k$, let $H_i(X, j)$ denote either rational motivic or \'etale homology, and $H^i(X, j)$ the corresponding cohomology.
        \begin{enumerate}[(i)]
            \item The homology $H_i(X, j)$ is covariant in $X$ with respect to proper morphisms, and contravariant with respect to \'etale morphisms. The cohomology $H^i(X, j)$ is contravariant in $X$.
            \item For $0 \le i \le \dim X$, there is a canonical isomorphism $H_{2i}^{\mc{M}}(X, \Q(i)) \cong \CH_i(X)$ \cite[Lemma 6.12.4]{Jannsen1990}. For $i, j \in \Z$, there is a regulator map
            \begin{gather*}
                H_{i}^{\mc{M}}(X, \Q(j)) \otimes_{\Q} \Q_\ell \to H_{i}(X_{\ol{k}}, \Q_\ell(j))
            \end{gather*}
            \cite[Example 8.3]{Jannsen1990}.
            \item There is a cap product
            \begin{gather*}
                \cap: H_i(X, j) \otimes H^m(X, n) \to H_{i-m}(X, j-n)
            \end{gather*}
            compatible with contravariance for \'etale morphisms. For a proper morphism $f: Y \to X$, the cap product satisfies the projection formula
            \begin{gather*}
                f_*(\alpha \cap f^*(\beta)) = f_*(\alpha) \cap \beta.
            \end{gather*}
            \item For $X$ irreducible of dimension $d$, there is a fundamental class $\eta_X \in H_{2d}(X, d)$, functorial with respect to contravariance for \'etale morphisms. When $X$ is additionally smooth, cap product with $\eta_X$ is an isomorphism
            \begin{gather}\label{eqPoinDual}
                H^{2d-i}(X, d-j) \cong H_i(X, j)
            \end{gather}
            compatible with regulator maps. Cohomology is thus covariant for proper morphisms of smooth varieties. These properties naturally extend to non-irreducible smooth schemes via
            \begin{gather*}
                H^i(X \sqcup Y, j) \cong H^i(X, j) \oplus H^i(Y, j).
            \end{gather*}
            For $X$ smooth and irreducible, and $f: Y \to X$ finite \'etale of degree $m$, we have $f_* (\eta_Y) = m \cdot \eta_X$, so the projection formula gives $f_*(f^*(\alpha)) = m \cdot \alpha$ for any $\alpha \in H^i(X, j)$.
            \item For $Z \subseteq X$ closed, with open complement $U = X - Z$, there is a long exact sequence
            \begin{gather}\label{eqHomLes}
                ... \to H_i(Z, j) \to H_{i}(X, j) \to H_{i}(U, j) \to H_{i-1}(Z, j) \to...
            \end{gather}
            compatible with regulator maps.
        \end{enumerate}
    \end{prop}

    In terms of motivic and \'etale homology, Jannsen generalized conjectures of Tate \cite[Conjecture $T^j(X)$]{Tate1994} and Beilinson \cite[page 3]{Beilinson1987} to arbitrary finite-type separated schemes over $\F_q$.
    \begin{conj}[$\TB_{i, j}(X)$: {\cite[Conjectures 12.4a and 12.6b]{Jannsen1990}}]\label{conjGenBeilinson}
        For $X$ a finite-type separated scheme (not necessarily quasi-projective) over a finite field $\F_q$, and for $i, j \in \Z$,
        \begin{enumerate}[(i)]
            \item the regulator map induces an isomorphism
            \begin{gather*}
            H_{i}^{\mc{M}}(X, \Q(j)) \otimes_{\Q} \Q_\ell \xrightarrow{\sim} H_{i}(X_{\ol{\F}_q}, \Q_\ell(j))^{G_{\F_q}};
            \end{gather*}
            \item the $\ell$-adic homology $H_{i}(X_{\ol{\F}_q}, \Q_\ell(j))$ is 1-semisimple.
        \end{enumerate}
    \end{conj}
    Here we define a finite-dimensional $G_{\F_q}$-representation $V$ to be \textbf{1-semisimple} if 1 is a semisimple eigenvalue of the geometric Frobenius $\Frob_q \in G_{\F_q}$ acting on $V$. We denote the truth value of the conjecture for a given $(X, i, j)$ by $\TB_{i, j}(X)$. By $\TB(X)$ we mean the truth value of the conjecture for a given $X$ and all $i, j \in \Z$. When $X$ is of pure dimension $d$, we define $\TB^{i, j}(X) = \TB_{2d-i, d-j}(X)$ (so the notation is compatible with Poincar\'e duality for a smooth $X$). The conjecture $\TB$, taken for all smooth projective varieties over $\F_q$, simultaneously generalizes the Tate--Beilinson and Parshin conjectures \cite[Conjecture 51 and Definition 55]{Kahn2005}. See loc. cit. for an overview of the general consequences of these conjectures.
    \par
    We now give some sufficient conditions for $\TB_{i, j}(X)$ (and related statements) to hold.
    \begin{prop}[{\cite[Proposition 3.1 and Lemma 3.6]{Broe2025}}]\label{propHomologyExSeq}
        For $X$ a finite-type, separated scheme over $\F_q$, $Z \subseteq X$ a closed subscheme, $U = X - Z$ the open complement, and $j \in \Z$, there is a commutative diagram with exact rows
        \[
            \begin{tikzcd}
                \CH_j(Z) \otimes_\Q \Q_\ell \arrow[r] \arrow[d] & \CH_j(X) \otimes_\Q \Q_\ell \arrow[r] \arrow[d] & \CH_j(U) \otimes_\Q \Q_\ell \arrow[r, overlay] \arrow[d] & 0 \\
                H_{2j}(Z_{\ol{\F}_q}, \Q_\ell(j))^{\Frob_q \sim 1} \arrow[r] & H_{2j}(X_{\ol{\F}_q}, \Q_\ell(j))^{\Frob_q \sim 1} \arrow[r] & H_{2j}(U_{\ol{\F}_q}, \Q_\ell(j))^{\Frob_q \sim 1} \arrow[r, overlay] & 0.
            \end{tikzcd}
        \]
        Here the superscript $\Frob_q \sim 1$ denotes the generalized eigenspace of the eigenvalue 1 of $\Frob_q$. In particular, $\TB_{2j, j}(Z)$ and $\TB_{2j, j}(X)$ together imply $\TB_{2j, j}(U)$.
    \end{prop}

    \begin{prop}\label{propWeakLefschetz}
        Let $X$ be projective over $\F_q$, and let $Z \subseteq X$ be a closed subscheme, with $U = X - Z$ smooth, affine, and of pure dimension $d$. 
        \begin{enumerate}[(i)]
            \item The weak Lefschetz theorem holds for \'etale homology: for any $j \in \Z$, the natural map
            \begin{gather*}
                H_{i}(Z_{\ol{\F}_q}, \Q_\ell(j)) \to H_{i}(X_{\ol{\F}_q}, \Q_\ell(j))
            \end{gather*} 
            is an isomorphism for $i < d - 1$, and surjective for $i = d - 1$. 
            \item In particular, if $2j < d-1$ and $\dim_{\Q_\ell} H_{2j}(X_{\ol{\F}_q}, \Q_\ell(j))^{\Frob_q \sim 1} = 1$, then the cycle class map
            \begin{gather*}
                \CH_j(Z) \otimes_{\Q} \Q_\ell \to H_{2j}(Z_{\ol{\F}_q}, \Q_\ell(j))^{\Frob_q \sim 1} 
            \end{gather*}
            is surjective.
        \end{enumerate} 
    \end{prop}
    \begin{proof}
        By Artin vanishing \cite[Tag 0F0W]{stacks-project} and Poincar\'e duality \eqref{eqPoinDual}, we have
        \begin{gather*}
            H_{i}(U_{\ol{\F}_q}, \Q_\ell(j)) \cong H^{2d-i}(U_{\ol{\F}_q}, \Q_\ell(d-j)) = 0
        \end{gather*}
        for $i < d$. Claim (i) now formally follows from the long exact sequence \eqref{eqHomLes}. Fixing a closed immersion $Z \to \mathbb{P}^n$, claim (ii) is immediate from (i) and the fact that any nonzero effective $j$-cycle in $Z$ is numerically nontrivial in $\mathbb{P}^n$.
    \end{proof}

    \begin{prop}\label{propTBCases}
        Let $X$ be quasi-projective over $\F_q$.
        \begin{enumerate}[(i)]
            \item If $Z$ is a closed subscheme of $X$ with open complement $U$, and $\TB$ holds for any two of $X$, $Z$ and $U$, then it also holds for the third one \cite[Theorem 12.7 b)]{Jannsen1990}.
            \item If $X$ is smooth, $L / \F_q$ is a finite extension, and $\TB_{i, j}(X_L)$ holds (where $X_L$ is considered as a scheme over $L$), then $\TB_{i, j}(X)$ holds.
            \item If $X$ is a smooth projective variety with finite-dimensional motive, the Tate conjecture holds for $X$, and for all $j$, $H^{2j}(X_{\ol{\F}_q}, \Q_\ell(j))$ is 1-semisimple, then $\TB(X)$ holds (\cite[Corollary 13.5.10]{Jannsen2007}).
            \item If $X$ is a smooth projective variety, rational and numerical equivalence coincide on $\CH^j(X)$, and the Tate conjecture holds for cycles of codimension $j$ in $X$, then $\TB^{2j, j}(X)$ holds \cite[Theorem 2.9]{Tate1994}.
        \end{enumerate}
    \end{prop}

    \begin{proof}
        Only claim (ii) requires a proof. The argument of \cite[Lemma 3.8]{Broe2025} shows that the $1$-semisimplicity of $H_i(X_{\ol{L}}, \Q_\ell(j))$ (considered as a $G_L$-representation) implies that of $H_i(X_{\ol{\F}_q}, \Q_\ell(j))$. For the injectivity of the cycle class map, let $f: X_L \to X$ be the natural morphism, and consider the commutative diagram of $\Q_\ell$-vector spaces
        \[
            \begin{tikzcd}
                H_i^{\mc{M}}(X, \Q(j)) \otimes_\Q \Q_\ell \arrow[r] \arrow{d}{f^*} & H_i(X_{\ol{\F}_q}, \Q_\ell(j))^{G_{\F_q}} \arrow{d}{f^*} \\
                H_i^{\mc{M}}(X_L, \Q(j)) \otimes_\Q \Q_\ell \arrow[r] & H_i(X_{\ol{L}}, \Q_\ell(j))^{G_{L}}.
            \end{tikzcd}
        \]
        The left arrow is injective by the projection formula (proposition \ref{propHomologyProps}(iv)), and the bottom arrow is injective by $\TB_{i, j}(X_L)$, so the top arrow is injective.
        \par
        For the surjectivity, consider
        \[
            \begin{tikzcd}
                H_i^{\mc{M}}(X, \Q(j)) \otimes_\Q \Q_\ell \arrow[r] & H_i(X_{\ol{\F}_q}, \Q_\ell(j))^{G_{\F_q}}  \\
                H_i^{\mc{M}}(X_L, \Q(j)) \otimes_\Q \Q_\ell \arrow{u}{f_*} \arrow[r] & H_i(X_{\ol{L}}, \Q_\ell(j))^{G_{L}} \arrow{u}{f_*}.
            \end{tikzcd}
        \]
        The bottom arrow is surjective by $\TB_{i, j}(X_L)$, and the right arrow is surjective by the projection formula, so the top arrow is surjective.
    \end{proof}
    We remark that if $X$ is a smooth projective variety over $\F_q$, then rational and numerical equivalence coincide on $\CH^j(X)$ when $j = 1$ (by the theory of the Picard variety and \cite[Lemma 2.3]{Broe2025}) or when $j = \dim X$ (by the theory of the Albanese variety and proposition \ref{propAlbKerFin}).

    \begin{cor}\label{corTBQuadric}
        The conjecture $\TB(X)$ holds for any smooth quadric hypersurface $X$ over $\F_q$.
    \end{cor}
    \begin{proof}
        Using proposition \ref{propTBCases}(ii), we may assume that $X$ has a rational point. The claim then follows from \cite[Proposition 70.1]{EKM2008} and induction on the dimension of $X$.
    \end{proof}

    \begin{cor}\label{corTBCors}
        The conjecture $\TB(X)$ holds for a smooth projective variety $X$ over $\F_q$ in the following cases:
        \begin{enumerate}[(i)]
            \item $X$ has motive of abelian type and satisfies the Tate conjecture;
            \item $X$ has dimension $\le 3$, and there exists a dominant rational map $Y \dashrightarrow X$ from a smooth projective variety $Y$, such that $Y$ has motive of abelian type, and satisfies the Tate conjecture for divisors.
        \end{enumerate}
    \end{cor}
    \begin{proof}
        Claim (i) is clear from proposition \ref{propTBCases}(iii) and the semisimplicity of the $\ell$-adic cohomology of abelian varieties over finite fields, which follows from \cite[page 138]{Tate1966} and Chevalley's theorem \cite[Theorem 1]{Maculan2018}.
        \par
        In the situation of (ii), by \cite[Theorem 3.11 and Proposition 3.13]{Vial2017}, there exists a direct sum decomposition
        \begin{gather*}
            \mf{h}(X) \cong Q \oplus R(-1),
        \end{gather*}
        where $Q$ is a direct summand of $\mf{h}(Y) \oplus \mf{h}(C_1)\oplus...\oplus \mf{h}(C_m)$ for some curves $C_1,...,C_m$, and $R$ is a direct summand of $\mf{h}(S_1) \oplus...\oplus \mf{h}(S_n)$ for some varieties $S_1,...,S_n$ of dimension $\le 2$\footnote{Note that \cite{Vial2017} uses a definition of \enquote{variety} that does not require irreducibility.}. Theorem 3.12 of loc. cit. additionally implies that the motive of $X$ is finite-dimensional. Using proposition \ref{propTBCases}(ii), we are free to extend scalars and assume that all $C_i$ and $S_j$ are geometrically connected. It then follows from the decomposition of the motive that $X$ satisfies the Tate conjecture for divisors, and that $H^{2j}(X_{\ol{\F}_q}, \Q_\ell(j))$ is 1-semisimple for $j \le 1$. The full Tate conjecture for $X$, and the 1-semisimplicity of its cohomology in higher degrees, then follows from the hard Lefschetz theorem \cite[Th\'eor\`eme (4.1.1)]{Deligne1980}. We conclude by proposition \ref{propTBCases}(iii).
    \end{proof}

    \begin{cor}\label{corTBFermat}
        The conjecture $\TB(X)$ holds when $X \subset \mathbb{P}^{n+1}_{\F_q}$ is a Fermat hypersurface of degree $m$, and either
        \begin{enumerate}[(i)]
            \item $p = \character(\F_q)$ satisfies $p^r \equiv -1 \pmod{m}$ for some $r$, or
            \item $n$ is odd.
        \end{enumerate}
    \end{cor}
    \begin{proof}
        Such an $X$ has motive of abelian type by \cite[Theorem 1]{KS79}. It also satisfies the Tate conjecture, by \cite[page 102]{Tate1965} in case (i), and proposition \ref{propFd} in case (ii). The claim thus follows from corollary \ref{corTBCors}.
    \end{proof}

    \begin{cor}\label{corTBCubic}
        The conjecture $\TB(X)$ holds for any smooth cubic hypersurface $X$ over $\F_q$ of dimension $\le 3$.
    \end{cor}
    \begin{proof}
        The claim is clear for $\dim X = 1$, by proposition \ref{propTBCases}(iii). Assume that $\dim X > 1$. By proposition \ref{propTBCases}(ii), we may pass to a finite extension and assume that $X$ has a rational point. Then by \cite[Theorem 1]{Kollar2002}, $X$ is unirational, so corollary \ref{corTBCors} applies.
    \end{proof}

    \begin{prop}\label{propTBK3AndCubic}
        Let $X$ be a smooth cubic hypersurface over $\mathbb{F}_q$ of dimension $d$. Then $\TB^{2j, j}(X)$ holds for $j \in \{1, 2, d - 1, d\}$.
    \end{prop}
    \begin{proof}
        By corollary \ref{corTBCubic}, we may assume that $d \ge 4$. The claim then follows from combining proposition \ref{propTBCases}(iv) with the results of section 2 on Chow groups of cubic hypersurfaces. Specifically, when $j = 1$, we use the Tate conjecture for divisors on $X$; when $j = 2$, we use the agreement of rational and numerical equivalence on $\CH^2(X)$ (proposition \ref{propRatNumCH2}) and the Tate conjecture in codimension 2 for $X$ (propositions \ref{propTateForK3Cubic} and \ref{propFd}); when $j = d-1$ we argue similarly, via corollary \ref{corCHDim1Cubic}, and when $j = d$ we use the finiteness of the Albanese kernel (proposition \ref{propAlbKerFin}).
    \end{proof}
    \par
    Using the previous proposition, we can show that the Fermat cubic fourfold sits in a larger family of cubic fourfolds satisfying the conjecture $\TB$.
    \begin{cor}\label{corTBForCubicFour}
        Assume that $\character(\F_q) \neq 3$. Let $X$ be a smooth cubic fourfold over $\F_q$, defined in $\mathbb{P}^5_{\F_q}$ by an equation of the form
        \begin{gather*}
            f(x_0,...,x_4) + x_5^3 = 0,
        \end{gather*}
        where $f(x_0,...,x_4) = 0$ defines a smooth cubic threefold in $\mathbb{P}^4_{\F_q}$. Then $\TB(X)$ holds.
    \end{cor}
    \begin{proof}
        As noted in remark \ref{remFdList}, it follows from the main result of \cite{Laterveer2018}, and a specialization argument as in the proof of corollary \ref{corH5Cubic}, that $X$ has finite-dimensional motive. By proposition \ref{propTBK3AndCubic}, $\TB^{2j, j}(X)$ holds for all $j$, and proposition \ref{propTBCases}(iii) then implies $\TB(X)$.
    \end{proof}
    We note that Rapoport proved in general that the $\ell$-adic cohomology of a smooth cubic fourfold over $\F_q$ is a semisimple Galois representation (\cite{Rapoport1972}; see also \cite[page 45]{Zarhin2004}).
    \par
    Using Voevodsky's proof of the Bloch--Kato conjecture \cite{Voe2011}, one can get some limited control over the integral motivic cohomology of varieties satisfying the conjecture $\TB$.
    \begin{prop}\label{propMotCohFinGen}
        Let $X$ be a smooth projective variety over $\F_q$ such that $\TB(X)$ holds, and let $i, j \in \Z$. Then if either
        \begin{enumerate}[(i)]
            \item $\dim X \le 3$, or
            \item $i \le j + 2$,
        \end{enumerate}
        then $H^i_{\mc{M}}(X, \Z(j))$ is a finitely generated abelian group.
    \end{prop}
    \begin{proof}[Sketch of proof]
        This follows from the argument of Corollaries 69 and 70 of \cite{Kahn2005}. We will fill in some details for the reader who wishes to check them, referring to the statements of loc. cit. as numbered therein. The crucial point is to observe, from the equivalent formulation of the Bloch--Kato conjecture in Lemma 18, that $H^i_{\mc{M}}(X, \mathbb{Z}(j))$ injects into $H^i_{\et}(X, \mathbb{Z}(j))$ for $i \le j + 2$. Here the latter group is \'etale motivic cohomology. Now use the finite generation of the Weil-\'etale cohomology of $X$ (proved in Theorem 67) and the long exact sequence of (4.39). Recalling that $H^i_{\mc{M}}(X, \Q(j)) \cong H^i_{\et}(X, \Q(j))$ (Theorem 14), one finds that $H^i_{\et}(X, \Z(j))$ can only fail to be finitely generated when $i = 2j + 2$. From this one concludes that $H^i_{\mc{M}}(X, \Z(j))$ is finitely generated whenever $i \le j + 2$ and $j > 0$. The lower bound on $j$ is removed through the straightforward observation that when $j \le 0$, one has 
        \begin{gather*}
            \CH^j(X, n)_\Z = \begin{cases}
                \Z & j = n = 0 \\
                0 & \textrm{ otherwise},
            \end{cases}
        \end{gather*}
        cf. \cite[page 268]{Bloch1986}.
    \end{proof}

    \subsection{The Beilinson--Bloch conjecture}\label{ssecBB}
    For $X$ a smooth projective variety over a global field $k$, the $\ell$-adic cohomology groups of $X$ have associated Hasse--Weil $L$-functions. When $\character(k) > 0$, these functions are rational by the Grothendieck--Lefschetz trace formula, and when $k$ is a number field, they are conjectured to have meromorphic continuation to $\mathbb{C}$.
    \begin{conj}[$\BB^i(X)$]\label{conjBB}
        For an integer $i$ with $0 \le i \le \dim X$,
        \begin{gather*}
            \dim_\Q \CH^i(X)_0 = \ord\limits_{s=i} L(H^{2i-1}(X_{\ol{k}}, \Q_\ell), s).
        \end{gather*}
    \end{conj}
    This conjecture was independently formulated over number fields by Beilinson \cite[Conjecture 5.0]{Beilinson1987} and Bloch \cite[page 94]{Bloch1984} (and both attributed it to Swinnerton-Dyer, who seems not to have published it). We write $\BB^i(X)$ for the truth value of the conjecture for a given $X$ and $i$. Sometimes included along with conjecture \ref{conjBB} are the statements that the Chow group with integer coefficients $\CH^i(X)_\Z$ is finitely generated (the \enquote{motivic Bass conjecture})\footnote{Note that Kahn expresses some doubt about the analogous assertion over finite fields in \cite[Section 4.9.1]{Kahn2005}.}, and that the Abel--Jacobi map
    \begin{gather*}
        \CH^i(X)_0 \otimes_{\Q} \Q_\ell \to H^1_{\cont}(G_k, H^{2i-1}(X_{\ol{k}}, \Q_\ell(i)))
    \end{gather*}
    is injective \cite[Conjecture 9.15]{Jannsen1990}. Here $H^1_{\cont}$ denotes continuous group cohomology. 
    \par
    Little is known in general about the Beilinson--Bloch conjecture, and we will not attempt more than a very limited recollection of existing results (see \cite[Remark 1.3]{LTXZZ2025} for some additional background). When $k = \Q$ or when $\character(k) > 0$, the conjecture is known for elliptic curves over $k$ of analytic rank at most one: this claim over $\Q$ follows from \cite{GZ1986}, \cite{Koly1988}, \cite{Wiles1995}, and \cite{BCDT01}, and over function fields from \cite[Theorems 12.1.2 and 12.3.1]{Qiu2022}. When $k$ is a number field and $i > 1$, partial evidence has come from the study of Ceresa cycles (as in \cite{Bloch1984}). The conjecture $\TB$, taken for all varieties over finite fields, is known to imply the Beilinson--Bloch conjecture in positive characteristic (\cite[Chapter 12]{Jannsen1990}, \cite[Corollary 4.15]{Broe2025}), and one can use this approach to prove the conjecture unconditionally in special cases (e.g. in the cases mentioned in the introduction, and in \cite[Theorem 3.9.1]{Ulmer2011}). In particular, we have the following criterion, which will be our main tool in proving cases of the Beilinson--Bloch conjecture later in the paper.
    \begin{prop}[{\cite[Remark 4.14 and Corollary 4.15]{Broe2025}}]\label{propBBCrit}
        Let $f: X \to C$ be a flat, projective morphism to a smooth, projective, geometrically connected curve $C$ over a finite field $\F_q$. Assume that the generic fiber $X_{k(C)}$ is a smooth variety over the function field of $C$. Let $Z$ be a finite nonempty set of closed points of $C$, and $U = C - Z$, such that $f$ is smooth over $U$. Let $d = \dim X$ and let $0 \le j \le d - 1$. If the cycle class map
        \begin{gather*}
            \CH_{d-j}(X_Z) \to H_{2d-2j}((X_{Z})_{\ol{\F}_q}, \Q_\ell(d-j))
        \end{gather*}
        surjects onto the Galois-invariant subspace, its target $H_{2d-2j}((X_{Z})_{\ol{\F}_q}, \Q_\ell(d-j))$ is 1-semisimple, and $\TB^{2j, j}(X)$ holds, then $\BB^j(X_\eta)$ holds, and the Abel--Jacobi map
        \begin{gather}\label{eqAJ}
            \CH^j(X_\eta)_0 \otimes_\Q \Q_\ell \to H^1_{\cont}(G_{k(C)}, H^{2j-1}(X_{\ol{k(C)}}, \Q_\ell(j)))
        \end{gather}
        is bijective. 
    \end{prop}
    We remark that it would follow from the Beilinson--Bloch conjecture that the Albanese kernel of any smooth, projective, geometrically connected variety over a global field is torsion \cite[Lemma 5.1]{Beilinson1987}. 
    \par
    The Beilinson--Bloch conjecture, in the form stated here, is closely related to conjectures of Beilinson, Bloch and Murre describing filtrations on the Chow groups of smooth projective varieties over arbitrary fields. We refer to Jannsen's article \cite{Jannsen1994} for a wonderful overview of some of the ideas behind and connections between the conjectures.

    \section{Monodromy of Lefschetz pencils}
    Here we compile some needed results from \cite{SGA7-II} and \cite{Deligne1980} on the monodromy of Lefschetz pencils, and use them to compute the Hasse--Weil $L$-function arising from the middle cohomology of a Lefschetz pencil over a finite field. This will later shed light on the arithmetic of the intermediate Jacobians of certain cubic threefolds (proposition \ref{propArithA}). The content of this section is applied primarily in section 7.
    \par
    Let $k$ be a field and $X$ a smooth, projective, geometrically connected variety over $k$ of dimension $d \ge 1$. We fix a closed immersion $X \subset \mathbb{P}^r_k$\footnote{In concrete examples, the specific closed immersion we consider will typically be clear from context.}. Let $\check{\mathbb{P}}^r_k$ denote the dual projective space parametrizing hyperplanes in $\mathbb{P}^r_k$, and let $\Gr(1, \check{\mathbb{P}}^r_k)$ denote its Grassmannian of lines. A line $h \in \Gr(1,\check{\mathbb{P}}^r)(k)$ has an associated axis $A(h) = \bigcap\limits_{H \in h} H$, which is a codimension-2 linear subspace in $\mathbb{P}^r_k$.
    \begin{defn}[{\cite[\hphantom{}XVII.2.2]{SGA7-II}}]
        A line $h \in \Gr(1, \check{\mathbb{P}}^r)(\ol{k})$ is called a \textbf{Lefschetz pencil} for $X_{\ol{k}}$ if $h$ satisfies the following conditions:
        \begin{enumerate}[(i)]
            \item the axis of $h$ intersects $X_{\ol{k}}$ transversely;
            \item there exists a dense open $U \subseteq \mathbb{P}^1_{\ol{k}}$ such that for all $t \in U(\ol{k})$, the hyperplane $h(t)$ intersects $X_{\ol{k}}$ transversely;
            \item for all $t \in \mathbb{P}^1(\ol{k}) - U(\ol{k})$, $h(t)$ intersects $X_{\ol{k}}$ transversely except at a single point, which is an ordinary double point singularity of $X_{\ol{k}} \cap h(t)$ \cite[D\'efinition XV.1.2.1]{SGA7-II}. We also refer to such a singularity as a node.
        \end{enumerate}
        More broadly, we say that $h \in \Gr(1, \check{\mathbb{P}}^r)(k)$ is a Lefschetz pencil for $X$ if it is a Lefschetz pencil for $X_{\ol{k}}$ in the above sense. We say that $X_{\ol{k}} \subset \mathbb{P}^{r}_{\ol{k}}$ is a \textbf{Lefschetz embedding} if $\Gr(1, \check{\mathbb{P}}^{r}_{\ol{k}})$ contains an open dense subvariety whose closed points consist of all Lefschetz pencils for $X_{\ol{k}}$.
    \end{defn}
    \begin{rem}
        An ordinary double point singularity of a variety $Z$ over $k$ is non-degenerate in the sense of \cite[D\'efinition XV.1.2.2]{SGA7-II} iff either $\dim Z$ is odd, or $\character(k) \neq 2$. Several of the results we will state have similar hypotheses on dimension and characteristic, which are partly explained by the need to address the edge case of degenerate ordinary double points.
    \end{rem}
    We proceed to recall a criterion for the existence of Lefschetz pencils (over $\ol{k}$). Let $N$ be the normal bundle of $X$ in the ambient projective space. Then the associated projective bundle\footnote{Note that \cite[Expos\'e XVII]{SGA7-II} uses the dual convention for projective bundles, i.e. for a vector bundle $V$, it defines $\mathbb{P}(V)$ to be the bundle of lines in $V^\vee$. We use the convention of \cite[Tag 01OA]{stacks-project}.} $\mathbb{P}(N)$ embeds into $X \times \check{\mathbb{P}}^r_k$, and its image consists of pairs $(x, H)$, with $H$ a hyperplane tangent to $X$ at $x$. The image of the projection $\phi: \mathbb{P}(N) \to \check{\mathbb{P}}^r_k$ is called the \textbf{dual variety} of $X$, denoted by $\check{X}$ \cite[\hphantom{}XVII.3.1.3]{SGA7-II}.
    \begin{lem}[{\cite[Corollaire XVII.3.5.0]{SGA7-II}}]\label{lemLefschetzCrit}
        Assume that $d$ is even, or that $\character(k) \neq 2$. Then $X_{\ol{k}} \subset \mathbb{P}^{r}_{\ol{k}}$ is a Lefschetz embedding if and only if either $\dim \check{X} \le r - 2$, or the extension of function fields $k(\mathbb{P}(N)) / k(\check{X})$ is finite separable.
    \end{lem}
    \begin{cor}
        Assume that $d$ is even, or that $\character(k) \neq 2$. If $X$ is a smooth hypersurface of degree $e$ in $\mathbb{P}^r_k$, and $\character(k) \nmid e(e-1)$, then $X_{\ol{k}} \subset \mathbb{P}^r_{\ol{k}}$ is a Lefschetz embedding.
    \end{cor}
    \begin{proof}
        Since $X$ is a smooth hypersurface, the normal bundle $N$ of $X$ has rank one, and $\mathbb{P}(N) \cong X$. By lemma \ref{lemLefschetzCrit}, we may assume that $\dim \check{X} = r-1 = \dim X$. Then the extension of function fields $k(X) / k(\check{X})$ induced by $\phi$ is finite. Again by the lemma, it suffices to show that $p = \character(k)$ does not divide the degree $[k(X): k(\check{X})]$, which we denote by $\deg(\phi)$. 
        \par
        Since $\dim \check{X} = r-1$, $\check{X}$ is a hypersurface in $\check{\mathbb{P}}^r_k$. Katz in \cite[Formula (XVII.5.2.8)]{SGA7-II} gives the following formula for its degree: we have
        \begin{gather*}
            \deg(\check{X}) \deg(\phi) = \deg ((\xi - [H])^{r-1}), 
        \end{gather*}
        where $\xi$ is the divisor of the normal bundle $N$ on $X$, $H$ is a hyperplane section of $X$, and $\deg ((\xi - [H])^{r-1})$ denotes the degree of a zero-cycle on $X$. Because $X$ is a smooth hypersurface of degree $e$ in $\mathbb{P}^r_k$, we have $\xi = e[H]$, and 
        \begin{gather*}
            \deg(\check{X}) \deg(\phi) = \deg(((e-1)[H])^{r-1}) = (e-1)^{r-1} \deg([H]^{r-1}) = (e-1)^{r-1} e.
        \end{gather*}
        By our hypothesis that $p \nmid e (e-1)$, we see that $p \nmid \deg(\phi)$, as desired.
    \end{proof}
    Fix a Lefschetz pencil $h \in \Gr(1, \check{\mathbb{P}}^r)(k)$ for $X$, and let $\Delta = \Delta(h) = X \cap A(h)$ be the base locus of the pencil. The blowup $Y$ of $X$ along $\Delta$ is smooth, and admits a morphism $f: Y \to \mathbb{P}^1$ such that $f^{-1}(t) \cong X_{\ol{k}} \cap h(t)$ for all $t \in \mathbb{P}^1(\ol{k})$ \cite[XVIII.3.1]{SGA7-II}. We will slightly abuse terminology and also refer to $f$ as a Lefschetz pencil for $X$. Let $\eta$ be the generic point of $\mathbb{P}^1_k$, and $\ol{\eta} \to \eta$ a geometric generic point. Let $G_\eta = \Gal(\ol{\eta} / \eta)$, and let $V = H^{d-1}(Y_{\ol{\eta}}, \Q_\ell)$.
    \begin{prop}\label{propLeray}
        Assume that one of the following conditions holds:
        \begin{enumerate}[(i)]
            \item $d$ is odd;
            \item $d$ is even, and the pullback map $H^{d-1}(X_{\ol{k}}, \Q_\ell) \to V$ is not surjective.
        \end{enumerate}
        Then the Leray spectral sequence
        \begin{gather*}
            E_2^{i, j}: H^i(\mathbb{P}^1_{\ol{k}}, R^j f_{\ol{k}, *} \Q_\ell) \implies H^{i+j}(Y_{\ol{k}}, \Q_\ell)
        \end{gather*}
        degenerates at $E_2$, and we have canonical isomorphisms
        \begin{align*}
            E_2^{0, d-1} &\cong H^{d-1}(X_{\ol{k}}, \Q_\ell),\\
            E_2^{2, d-1} &\cong H^{d-1}(X_{\ol{k}}, \Q_\ell(-1)).
        \end{align*}
        compatible with the $G_k$-action. For $j \neq d-1$, we have canonical isomorphisms
        \begin{align*}
            E_2^{0, j} &\cong H^{j}(Y_{\ol{\eta}}, \Q_\ell),\\
            E_2^{2, j} &\cong H^{j}(Y_{\ol{\eta}}, \Q_\ell(-1)),\\
            E_2^{i, j} &= 0 \textrm{ for } i \notin \{0, 2\},
        \end{align*}
        compatible with the $G_\eta$-action through the quotient $G_\eta \to G_k$.
    \end{prop}
    \begin{proof}
        This follows from \cite[Th\'eor\`emes XVIII.5.6 and XVIII.6.3.3]{SGA7-II} and the hard Lefschetz theorem \cite[Th\'eor\`eme 4.1.1]{Deligne1980}.
    \end{proof}
    We will need some facts about vanishing cycles, which we now review. The reader is advised that this requires a bit of care with basepoints. Let $m = \left \lfloor \frac{d-1}{2} \right \rfloor$. Recall that for two geometric points $\ol{a}$, $\ol{b}$ of a scheme $Z$, a \textbf{path} in $Z$ from $\ol{a}$ to $\ol{b}$ is an isomorphism between the associated set-valued fiber functors on the category of finite \'etale covers of $Z$. Let $S$ be the set of points $\ol{x} \in \mathbb{P}^1_{\ol{k}}$ such that the fiber $Y_{\ol{x}}$ is singular, and $U = \mathbb{P}^1_{\ol{k}} - S$. For each $\ol{x} \in S$, fix a geometric generic point $\ol{\eta}_{\ol{x}}$ of the henselization of $\mathbb{P}^1_{\ol{k}}$ at $\ol{x}$. For each path $c$ in $U$ from $\ol{\eta}$ to $\ol{\eta}_{\ol{x}}$, there is an associated vanishing cycle $\delta_c \in V(m)$, well-defined up to sign. We define 
    \begin{gather*}
        E_{\ol{x}} = \{\pm \delta_c: c \textrm{ a path from } \ol{\eta} \textrm{ to } \ol{\eta}_{\ol{x}}\} \subseteq V(m),
    \end{gather*}
    following \cite[D\'efinition 4.2.4]{Deligne1980}. Let
    \begin{gather*}
        C = \bigcup_{\ol{x} \in S} E_{\ol{x}} \subseteq V(m),
    \end{gather*}
    and let $\Ev(Y_{\ol{\eta}}) \subseteq V$ be the subspace spanned by $\{\lambda \delta: \lambda \in \Q_\ell(-m), \delta \in C\}$.
    \begin{prop}[{\cite[Corollaire 4.3.9]{Deligne1980}}]\label{propVanSpan}
        The pullback map $H^{d-1}(X_{\ol{k}}, \Q_\ell) \to V$ induces a direct sum decomposition
        \begin{gather*}
            V \cong H^{d-1}(X_{\ol{k}}, \Q_\ell) \oplus \Ev(Y_{\ol{\eta}}).
        \end{gather*}
    \end{prop}
    Let $G_\infty$ be the Galois group of $\ol{k(\eta)}$ over $\ol{k}(\eta)$, and note that the $G_\infty$-action on $V$ factors through the map $G_\infty \to \pi_1(U, \ol{\eta})$, which is surjective by \cite[Tag 0BQI]{stacks-project}. Let $\rho: G_\infty \to \mc{G}$ denote the $G_\infty$-representation on $\Ev(Y_{\ol{\eta}})$. This lands in the group $\mc{G}$ of automorphisms of $\Ev(Y_{\ol{\eta}})$ which preserve the nondegenerate cup product pairing ${\Ev(Y_{\ol{\eta}}) \otimes_{\Q_\ell} \Ev(Y_{\ol{\eta}})} \to \Q_\ell(-(d-1))$.
    \begin{prop}[{\cite[Th\'eor\`eme 4.4.1]{Deligne1980}}]\label{propOpenImage}
        If $d$ is even, the image $\rho(G_\infty)$ is open in $\mc{G}$.
    \end{prop}
    \begin{rem}
        When $d$ is even, the cup product pairing on $\Ev(Y_{\ol{\eta}})$ is alternating, and $\mc{G}$ is the associated symplectic group (which is nonabelian if $\Ev(Y_{\ol{\eta}})$ is nonzero).
    \end{rem}
    \begin{prop}[{\cite[Corollaire 4.2.8]{Deligne1980}}]\label{propVanTran}
        Suppose that $d$ is even, or that $d$ is odd and $\character(k) \neq 2$. Then for all $\delta, \delta' \in C$, there exists $g \in G_\infty$ such that $g(\delta) \in \{\delta', -\delta'\}$. In particular, if one vanishing cycle is nonzero, then they all are.
    \end{prop}
    \par
    Fix $\ol{x} \in S$, and a path $c$ from $\ol{\eta}$ to $\ol{\eta}_{\ol{x}}$. Let $I_{c} \subset G_{\eta}$ be the associated inertia group. Recall the specialization map $\spe: H^{d-1}(Y_{\ol{x}}, \Q_\ell) \to H^{d-1}(Y_{\ol{\eta}_{\ol{x}}}, \Q_\ell)$ \cite[Tag 0GJ2]{stacks-project}, whose target is identified via $c$ with $V$. 
    \begin{prop}\label{propInertia}
        The specialization map is injective, and its image in $V$ is equal to $V^{I_{c}}$. This in turn is equal to the orthogonal complement of the vanishing cycle $\delta_{c}$ with respect to the cup product $V \otimes_{\Q_\ell} V(m) \to \Q_\ell(-(d - m - 1))$. If $d$ is even or $\character(k) \neq 2$, the action of $I_c$ on $V$ factors through the tame quotient of $I_c$.
    \end{prop}
    \begin{proof}
        The claims are immediate from \cite[Th\'eor\`eme XV.3.4]{SGA7-II}.
    \end{proof}

    Suppose now that $k = \F_q$ is finite. Under mild hypotheses, the above results permit us to explicitly compute the Hasse--Weil $L$-function of the $G_{\eta}$-representation $V$ in terms of the zeta function of $Y$. For a Galois representation $W$ of a finite field $\F_{a}$ with $a$ elements, let
    \begin{gather*}
        \chi(W, s) = \det(1-\Frob_a a^{-s} | W),
    \end{gather*}
    where $\Frob_a \in G_{\F_a}$ is the geometric Frobenius. Let $(\mathbb{P}^1_{k})^0$ denote the set of closed points of $\mathbb{P}^1_k$. For each $x \in (\mathbb{P}^1_{k})^0$, fix a choice of inertia and decomposition groups $I_{x} \subset G_{x} \subset G_{\eta}$, and regard $V^{I_x}$ as a $G_{k(x)}$-representation in the natural way. Let
    \begin{gather*}
        L(V, s) = \prod\limits_{x \in (\mathbb{P}^1_{\F_q})^0} \chi(V^{I_{x}}, s)^{-1}
    \end{gather*}
    be the Hasse--Weil $L$-function of $V$, and
    \begin{gather*}
        \Lambda(s) = \prod\limits_{x \in (\mathbb{P}^1_{\F_q})^0} \chi(H^{d-1}(Y_{\ol{x}}, \Q_\ell), s)^{-1}
    \end{gather*}
    be the $L$-function of the sheaf $R^{d-1} f_{*} \Q_\ell$ on $\mathbb{P}^1_k$. 

    \begin{prop}\label{propLefschetzL}
        For all $x \in (\mathbb{P}^1_k)^0$, we have
        \begin{gather*}
            \chi(V^{I_{x}}, s) = \chi(H^{d-1}(Y_{\ol{x}}, \Q_\ell), s),
        \end{gather*}
        so that $L(V, s) = \Lambda(s)$. If the hypotheses of proposition \ref{propLeray} are satisfied, then
        \begin{gather*}
            L(V, s) = \frac{\chi(H^{d}(Y_{\ol{k}}, \Q_\ell), s)}{\chi(H^{d-1}(X_{\ol{k}}, \Q_\ell), s) \ \chi(H^{d-1}(X_{\ol{k}}, \Q_\ell), s - 1) \ \chi(H^d(Y_{\ol{\eta}}, \Q_\ell), s) \ \chi(H^{d-2}(Y_{\ol{\eta}}, \Q_\ell), s - 1)}.
        \end{gather*}
    \end{prop}
    \begin{proof}
        For any $x \in (\mathbb{P}^1_k)^0$, proposition \ref{propInertia} gives a $G_{k(x)}$-equivariant identification
        \begin{gather*}
            H^{d-1}(Y_{\ol{x}}, \mathbb{Q}_\ell) \cong H^{d-1}(Y_{\ol{\eta}}, \Q_\ell)^{I_{x}},
        \end{gather*}
        so that the Euler factors of $L(V, s)$ and $\Lambda(s)$ at each closed point agree. By the Grothendieck--Lefschetz trace formula \cite[Theorem VI.13.3]{Milne1980},
        \begin{gather}\label{eqEulerProd}
            \Lambda(s) = \frac{\chi(H^1(\mathbb{P}^1_{\ol{k}}, R^{d-1} f_{\ol{k}, *} \Q_\ell), s)}{\chi(H^0(\mathbb{P}^1_{\ol{k}}, R^{d-1} f_{\ol{k}, *} \Q_\ell), s)\chi(H^2(\mathbb{P}^1_{\ol{k}}, R^{d-1} f_{\ol{k}, *} \Q_\ell), s)}.
        \end{gather}
        Now assume that the hypotheses of proposition \ref{propLeray} are satisfied, so that the Leray spectral sequence discussed there degenerates at $E_2$. We then have a direct sum decomposition
        \begin{gather*}
            H^d(Y_{\ol{k}}, \Q_\ell) \cong H^0(\mathbb{P}^1_{\ol{k}}, R^d f_{\ol{k}, *} \Q_\ell) \oplus H^1(\mathbb{P}^1_{\ol{k}}, R^{d-1} f_{\ol{k}, *} \Q_\ell) \oplus H^2(\mathbb{P}^1_{\ol{k}}, R^{d-2} f_{\ol{k}, *} \Q_\ell),
        \end{gather*}
        cf. \cite[Proposition 1.2]{Deligne1968}. The expressions for the $E_2$-terms given in proposition \ref{propLeray}, together with \eqref{eqEulerProd}, complete the proof.
    \end{proof}

    \section{Lefschetz pencils of some special varieties}\label{secLefschetzSpecial}

    Let $k = \F_q$, and again let $X \subset \mathbb{P}^r_k$ be a smooth, projective, geometrically connected variety of dimension $d$. Fix a Lefschetz pencil $h$ for $X$ which is defined over $k$. Again let $\Delta$ be the base locus of $h$, let $f: Y \to \mathbb{P}^1_k$ be the family of hyperplane sections of $X$ associated with $h$, and let $Y_\eta$ be the generic fiber of $f$. Our goal is to use the criterion of proposition \ref{propBBCrit} to prove cases of the Beilinson--Bloch conjecture for $Y_\eta$. This requires understanding the Chow groups of $Y$, and those of the singular closed fibers of $f$. We will give two similar arguments along these lines, first when $X$ is a cubic fourfold, then when $X$ is a more general cubic hypersurface. When $X$ is a threefold with $\CH^2(X)$ one-dimensional, we prove that the group of zero-cycles of degree zero of $Y_\eta$ is torsion through more elementary means. 
    \par
    Before treating more concrete situations, we include the following general sufficient condition for the Beilinson--Bloch conjecture for $Y_\eta$ in high codimension.

    \begin{prop}\label{propBBHighCodim}
        Fix $j$ with $2j > d + 1$. Suppose that $\TB^{2j, j}(X)$ and $\TB^{2(j-1), j-1}(\Delta)$ hold. Also assume that 
        \begin{gather*}
            \dim_{\Q_\ell} H^{2j}(X_{\ol{k}}, \Q_\ell(j))^{\Frob_{q}^r \sim 1} = 1
        \end{gather*}
        for all $r > 0$, where the superscript $\Frob_q^r \sim 1$ denotes the 1-generalized eigenspace of $\Frob_q^r$. Then $\BB^{j}(Y_\eta)$ holds. 
    \end{prop}
    \begin{proof}
        By the blowup formula for motives (proposition \ref{propBlowup}), $\TB^{2j, j}(X)$ and $\TB^{2(j-1), j-1}(\Delta)$ together are equivalent to $\TB^{2j, j}(Y)$. Let $Z \subset \mathbb{P}^1_k$ be a finite nonempty set of closed points, with $U = \mathbb{P}^1_k - Z$, such that $f_U$ is smooth. Using proposition \ref{propBBCrit}, it suffices to show that the cycle class map
        \begin{gather*}
            \CH_{d-j}(Y_Z) \to H_{2d-2j}((Y_{Z})_{\ol{k}}, \Q_\ell(d-j))^{\Frob_q \sim 1}
        \end{gather*}
        is surjective. The argument of proposition \ref{propTBCases} shows that it even suffices to prove the analogous statement after passing to a finite extension $L = \F_{q^r}$, i.e. that  
        \begin{gather*}
            \CH_{d-j}((Y_Z)_L) \to H_{2d-2j}((Y_{Z})_{\ol{k}}, \Q_\ell(d-j))^{\Frob_q^r \sim 1}
        \end{gather*}
        is surjective. Fix such an $L$ so that the points of $Z$ are $L$-rational. Then $(Y_Z)_L$ is a finite disjoint union of hyperplane sections of $X_L$, and the desired surjectivity follows from the weak Lefschetz theorem (proposition \ref{propWeakLefschetz}). 
    \end{proof}
    \begin{rem}
        For example, from proposition \ref{propAlbKerFin} and corollary \ref{corTBCors}, one sees that the hypotheses of proposition \ref{propBBHighCodim} are satisfied when $d \ge 4$, $j = d-1$, and $X$ is a Fermat hypersurface which satisfies the Tate conjecture. 
    \end{rem}
    Now we turn our focus to smooth cubic hypersurfaces. The following lemma characterizes the singular fibers of their Lefschetz pencils.
    \begin{lem}[{\cite[Corollary 1.5.16]{Huybrechts2023}}]\label{lemNodalCubic}
        If $Z$ is a cubic hypersurface over a field $L$ such that $Z$ is smooth away from a single $L$-rational ordinary double point $z$, then the blowup of $Z$ at $z$ is isomorphic to the blowup of $\mathbb{P}^{\dim Z}_L$ along a smooth complete intersection of multidegree (3, 2). 
    \end{lem}
    With this, we can prove the below theorem, which implies theorem \ref{thmBBCubicThreefoldIntro}.
    \begin{thm}\label{thmBBCubicThreefold}
        Assume that $X$ is a smooth cubic fourfold.
        \begin{enumerate}[(i)]
            \item The conjecture $\TB$ holds for all closed fibers of $f$.
            \item For all $j$, $\TB^{2j, j}(Y)$ holds, and the Beilinson--Bloch conjecture holds for codimension-2 cycles on $Y_\eta$. If $\TB(X)$ holds, then $\TB(Y)$ does as well.
        \end{enumerate}
    \end{thm}
    \begin{proof}
        Let $F$ be a fiber of $f$ over a closed point $t_0 \in \mathbb{P}^1_k$, $T$ the regular locus of $F$, and $\Sing(F) = F - T$ the singular locus. As $\Sing(F)$ is either empty or zero-dimensional, $\TB(\Sing(F))$ holds. Thus, by the \enquote{two-out-of-three} part of proposition \ref{propTBCases}, $\TB(T)$ is equivalent to $\TB(F)$. By proposition \ref{propTBCases}(ii), to prove $\TB(T)$, it suffices to prove $\TB(T_M)$ for some finite extension $M / \F_q$. To prove $\TB(F)$, we therefore can and do assume that $t_0$ is $\F_q$-rational, so $F$ is a cubic threefold over $\F_q$. If $F$ is smooth, then $\TB(F)$ holds by corollary \ref{corTBCubic}, so we are left to consider the case where $\Sing(F)$ consists of one $k$-rational ordinary double point.
        \par
        Let $\alpha: B \to F$ be the blowup of $F$ along its singular point. By lemma \ref{lemNodalCubic}, $B$ is rational, so $\TB(B)$ holds by corollary \ref{corTBCors}. The exceptional divisor $D$ of $\alpha$ is a smooth quadric surface, so $\TB(D)$ holds by corollary \ref{corTBQuadric}. Thus $\TB(F)$ holds by the two-out-of-three property.
        \par
        It remains to prove (ii). We noted in proposition \ref{propTBK3AndCubic} that $\TB^{2j, j}(X)$ holds for all $j$. Since $Y$ is the blowup of $X$ along a smooth cubic surface $\Delta$, and $\TB(\Delta)$ holds (corollary \ref{corTBCubic}), $\TB^{2j, j}(Y)$ also holds for all $j$. Then $\BB^2(Y_\eta)$ follows from part (i) and proposition \ref{propBBCrit}. If $\TB(X)$ holds, we similarly infer $\TB(Y)$.
    \end{proof}

    \begin{thm}\label{thmBB3Cubic}
        Suppose that $X$ is a smooth cubic hypersurface of dimension $d \ge 5$. Then $\TB^{4, 2}$ holds for all closed fibers of $f$. If $\TB^{6, 3}(X)$ holds (e.g. $X$ is a Fermat cubic hypersurface satisfying the hypotheses of corollary \ref{corTBFermat}), then $\BB^3(Y_\eta)$ does as well.
    \end{thm}
    \begin{proof}
        Let $F$ be a closed fiber of $f$. We want to prove $\TB^{4, 2}(F) = \TB_{2d-6, d-3}(F)$. As in the proof of theorem \ref{thmBBCubicThreefold}, by applying proposition \ref{propHomologyExSeq} to the open-closed decomposition of $F$ into its singular and regular locus, and using proposition \ref{propTBCases}(ii), we reduce to the case where $F$ is a cubic hypersurface over $\F_q$. 
        \par
        If $F$ is smooth, then $\TB^{4, 2}(F)$ holds by proposition \ref{propTBK3AndCubic}. We may thus assume that $F$ is smooth away from a single $k$-rational ordinary double point. With $\alpha: B \to F$ the blowup of the singular point, we see from proposition \ref{propFd} and lemma \ref{lemNodalCubic}, and the blowup formula for motives, that $\TB^{4, 2}(B)$ holds. In detail, we have
        \begin{gather*}
            \mf{h}(B) \cong \mf{h}(\mathbb{P}^{d-1}) \oplus \mf{h}(S)(-1),
        \end{gather*}
        where $S \subset \mathbb{P}^{d-1}$ is a smooth complete intersection of type (3, 2). Thus $\TB^{2,1}(S)$ implies $\TB^{4,2}(B)$. When $d = 5$, it is well-known that $S$ is a K3 surface \cite[Example 1.1.3]{Huybrechts2016}, so that $\TB^{2,1}(S)$ follows from the known Tate conjecture for $S$ \cite[Theorem A.1]{KM2016} and proposition \ref{propTBCases}(iv). When $d > 5$, we observe from the Chow--K\"unneth decomposition of proposition \ref{propFd} that $\TB^{2, 1}(S)$ holds.
        \par
        The exceptional divisor $D$ of $\alpha$ is a smooth quadric hypersurface, so $\TB(D)$ holds by corollary \ref{corTBQuadric}. Proposition \ref{propHomologyExSeq} then implies $\TB^{4, 2}(B- D)$. Using the diagram of that proposition with respect to the open-closed decomposition $F = \Spec \F_q \sqcup (B-D)$, we see that $\TB^{4,2}(F)$ holds. 
        \par
        Now assume $\TB^{6,3}(X)$. The blowup formula then shows that $\TB^{6,3}(Y)$ is equivalent to $\TB^{4,2}(\Delta)$, which in turn follows from proposition \ref{propTBK3AndCubic}. From $\TB^{6,3}(Y)$ and the preceding discussion of closed fibers of $f$, proposition \ref{propBBCrit} yields $\BB^3(Y_\eta)$, as desired.
    \end{proof}

    \begin{rem}\label{remEasyBBCubic}
        For $S$ a smooth cubic hypersurface of dimension $\ge 2$ over a global field, there are some cases of the Beilinson--Bloch conjecture for $S$ which are immediate from previously known results and the Betti numbers of $S$. For instance, $\BB^1(S)$ follows from the triviality of the Picard variety of $S$. When $d > 3$, $\BB^2(S)$ follows from proposition \ref{propRatNumCH2}. Likewise, $\BB^{\dim S}(S)$ follows from remark \ref{remCHdim0Cubic}. Other cases of $\BB^i(S)$ are implied by proposition \ref{propGenChVanish}. In particular, if we take $X$ to be a cubic fivefold over $\F_q$ which satisfies the hypotheses of theorem \ref{thmBB3Cubic}, we find that rational and numerical equivalence agree on the Chow groups (with rational coefficients) of the corresponding cubic fourfold $Y_\eta$ over $\F_q(t)$. 
    \end{rem}

    \begin{prop}\label{propCh21d3Fold}
        Suppose that $\dim(X) = 3$, and that $\CH^2(X)$ is one-dimensional (e.g. $X$ is a smooth Fermat threefold). Then $\CH^2(Y_\eta)$ is one-dimensional. In particular, the group of zero-cycles of degree zero $\CH^2(Y_\eta)_{0, \Z}$ is torsion.
    \end{prop}
    \begin{proof}
        As $Y$ is the blowup of $X$ along a smooth curve, the blowup formula for motives shows that $\CH^2(Y)$ is two-dimensional. Let $L, M \in \CH^1(Y)$, such that $M$ is the class of a closed fiber of $f$, and $L$ is ample. Then $L \cdot M$ is numerically nontrivial in $Y$. Since $L \cdot M$ is in the kernel of the surjection $\CH^2(Y) \to \CH^2(Y_\eta)$, its target (which is nonzero) must be one-dimensional.
    \end{proof}

    \begin{prop}\label{propCh2FinGen}
        Assume that $\TB(X)$ holds. Then $\CH^2(Y_\eta)_\Z$ is finitely generated.
    \end{prop}
    \begin{proof}
        Proposition \ref{propMotCohFinGen} implies that $\CH^2(X)_\Z$ is finitely generated, and the blowup formula for motives then implies the same for $\CH^2(Y)_\Z$. The latter surjects onto $\CH^2(Y_\eta)_\Z$.
    \end{proof}

    \begin{cor}\label{corCH2FiniteEx}
        If $\TB(X)$ holds, and $X$ is either a smooth threefold with $\CH^2(X)$ one-dimensional, or a smooth cubic hypersurface of dimension $\neq 4$, then the homologically trivial subgroup of $\CH^2(Y_\eta)_\Z$ is finite.
    \end{cor}
    \begin{proof}
        This follows from propositions \ref{propCh21d3Fold} and \ref{propCh2FinGen}, and the cases of the Beilinson--Bloch conjecture discussed in remark \ref{remEasyBBCubic}.
    \end{proof}

    \begin{rem}
        The statement of proposition \ref{propMotCohFinGen} was previously used in \cite[Proof of Corollary 1]{Kahn2021} to prove finiteness of the Albanese kernel of some isotrivial surfaces over global function fields.
    \end{rem}

    \begin{rem}\label{remNonIsotrivial}
        Suppose that at least one of the following conditions holds:
        \begin{enumerate}
            \item $X$ is even-dimensional, and the pullback map $H^{d-1}(X_{\ol{k}}, \Q_\ell) \to H^{d-1}(Y_{\ol{\eta}}, \Q_\ell)$ is not surjective;
            \item $X$ is a Fermat hypersurface of degree $m$, where $p = \character(k) \nmid m$ and $m-1$ is not a power of $p$, and the Lefschetz pencil $h$ is \enquote{sufficiently generic} (i.e. it lies outside a certain closed subset of $\Gr(1, \check{\mathbb{P}}^n)$, which depends on $X$).
        \end{enumerate}
        Then $Y_\eta$ is not isotrivial. In the former case, this follows from propositions \ref{propVanSpan} and \ref{propOpenImage}. In the latter case, it follows from \cite[Example 9]{OV2007}. Taking $X$ to be the Fermat quartic threefold in corollary \ref{corCH2FiniteEx}, we thus find that for $h$ sufficiently generic, $Y_\eta$ is a non-isotrivial quartic K3 surface over $\F_q(t)$ whose group of zero-cycles of degree zero is finite. As remarked in \cite[Example 9.3]{Huybrechts2014}, it would be interesting to find an example of a non-isotrivial K3 surface $Z$ over $\ol{\F_q(t)}$ with $\CH^2(Z)_0$ trivial.
    \end{rem}

    \subsection{Motivic cohomology of $Y_\eta$}
    In this subsection, we assume that $X$ is a smooth cubic fourfold over $k = \F_q$ satisfying $\TB(X)$: for instance, we may take $X$ as in corollary \ref{corTBForCubicFour}. Using theorem \ref{thmBBCubicThreefold} and results of Jannsen, we can compute all motivic cohomology groups of $Y_\eta$ (with rational coefficients) in terms of certain simpler invariants, whose definitions we now recall.
    \par
    Fix a geometric point $\ol{\eta} \to \eta$, and let $K = k(\eta)$. For any smooth variety $W$ over $k(\eta)$ and $b,$ $c \in \Z$, define the arithmetic \'etale cohomology
        \begin{gather}\label{eqArEtale}
            H^b_{\ar}(W, \Q_\ell(c)) = \left[\varinjlim\limits_{U' \subseteq U} H^b(\widetilde{W}_{U'} \times_{\F_q} \ol{\F}_q, \Q_\ell(c))\right]^{G_{\F_q}}.
        \end{gather}
        Here $\widetilde{W}$ is a choice of smooth proper scheme over a nonempty open $U \subseteq \mathbb{P}^1_k$ such that $\widetilde{W} \times_{U} K \cong W$, and the colimit is over all nonempty open $U' \subseteq U$. As explained in \cite[Section 13.6]{Jannsen2007}, the arithmetic \'etale cohomology receives a regulator map
        \begin{gather*}
            H^b_{\mc{M}}(W, \Q(c)) \to H^b_{\ar}(W, \Q_\ell(c)).
        \end{gather*}
        For $a$, $b$, $c \in \Z$, the choice of $\widetilde{W}$ likewise equips $V = H^b(W_{\ol{K}}, \Q_\ell(c))$ with an action of the \'etale fundamental group $\pi_1(U, \ol{\eta})$. We define the arithmetic Galois cohomology of $V$ via
        \begin{gather}\label{eqArGalois}
            H^a_{\ar}(G_K, V) = \left[\varinjlim\limits_{U' \subseteq U} H^a_{\cont}(\pi_1(U'_{\ol{\F}_q}, \ol{\eta}), V)\right]^{G_{\F_q}}.
        \end{gather}
        The arithmetic \'etale and Galois cohomology groups, and the regulator map, are independent of the choice of $\widetilde{W}$.
        \par
        Since we assumed $\TB(X)$, theorem \ref{thmBBCubicThreefold} implies that $\TB$ holds for $Y$, and for all closed fibers of $f$. Now \cite[Theorem 13.6.3]{Jannsen2007} yields the below result.
        \begin{thm}\label{thmBBM}
        The following are true.
        \begin{enumerate}[(i)]
            \item The regulator maps induce isomorphisms
            \begin{gather*}
                H^i_{\mc{M}}(Y_\eta, \Q(j)) \otimes_\Q \Q_\ell \to H^i_{\ar}(Y_\eta, \Q_\ell(j))
            \end{gather*}
            for all $i, j \in \Z$.
            \item For $2j-i \notin \{0, 1\}$, we have\footnote{This corrects a small typo in the statement of \cite[Theorem 13.6.3]{Jannsen2007}, which says that $H^i_{\mc{M}}(Y_\eta, \Q(j)) = 0$ for $i - 2j \notin \{0, 1\}$, instead of for $2j - i \notin \{0, 1\}$. As a subquotient of $K_{2j-i}(Y_\eta) \otimes_\Z \Q$, $H^i_{\mc{M}}(Y_\eta, \Q(j))$ automatically vanishes when $2j - i < 0$, since $Y_\eta$ is regular. Similarly, formula \eqref{eqArithCohIso} corrects another small misprint in loc. cit., which, after accounting for the previous correction, originally says that $H^{2j-1}_{\ar}(Y_\eta, \Q_\ell(j)) \cong H^{1}_{\ar}(G_K, H^{2j-1}(Y_{\ol{K}}, \Q_\ell(j)))$.} $H^i_{\mc{M}}(Y_\eta, \Q(j)) = 0$. For $i = 2j - 1$, there is an isomorphism
            \begin{gather}\label{eqArithCohIso}
                H^{2j-1}_{\ar}(Y_\eta, \Q_\ell(j)) \cong H^{1}_{\ar}(G_K, H^{2j-2}(Y_{\ol{\eta}}, \Q_\ell(j))).
            \end{gather}
            \item For all $j$, the cycle class map
            \begin{gather*}
                \CH^j(Y_\eta)_{\hom} \otimes_\Q \Q_\ell \xrightarrow{\sim} H^{2j}(Y_{\ol{\eta}}, \Q_\ell(j))^{G_K}
            \end{gather*}
            and Abel--Jacobi map
            \begin{gather*}
                \CH^j(Y_\eta)_0 \otimes_\Q \Q_\ell \xrightarrow{\sim} H^1_{\cont}(G_K, H^{2j-1}(Y_{\ol{\eta}}, \Q_\ell(j)))
            \end{gather*}
            are both isomorphisms.
        \end{enumerate}
        In particular, Jannsen's conjecture \cite[Conjecture 12.18]{Jannsen1990} holds for $Y_\eta$.
    \end{thm}
    Note that claim (iii) of the theorem also follows from propositions \ref{propFd} and \ref{propBBCrit}. We remark that Murre's conjecture \cite[\hphantom{} 1.4]{Murre1993}, which \cite[Theorem 13.6.3]{Jannsen2007} also implies for $Y_\eta$, was previously known for cubic threefolds \cite[Proof of Theorem 4.4]{dAMS98}.

    \section{The intermediate Jacobian}\label{secBSD}
    Let $X$ be a smooth cubic fourfold over $k = \F_q$, and $f: Y \to \mathbb{P}^1_k$ a Lefschetz pencil for $X$, with generic fiber $Y_\eta$. Then $Y_\eta$ is a smooth cubic threefold over $K = \F_q(t)$. Let $A$ be the Albanese variety of the Fano surface of $Y_\eta$, and fix a Chow--K\"unneth decomposition of $Y_\eta$ as in proposition \ref{propFd}. Passing to cohomology and Chow groups in proposition \ref{propH3Cubic} yields
    \begin{gather*}
        H^{3}(Y_{\ol{\eta}}, \Q_\ell(2)) \cong H^1(A_{\ol{\eta}}, \Q_\ell(1))
    \end{gather*}
    and
    \begin{gather*}
        \CH^2(Y_\eta)_0 = \CH^2(\mf{h}^3(Y_\eta)) = \CH^1(A)_0 \cong A(K) \otimes_\Z \Q,
    \end{gather*}
    where the last isomorphism is induced by a choice of polarization of $A$. It immediately follows that the Beilinson--Bloch conjecture for $Y_\eta$ (proved in theorem \ref{thmBBCubicThreefold}) is equivalent to the Birch and Swinnerton-Dyer conjecture for $A$. The latter is in turn equivalent via the main theorem of \cite{KT2003} to the strong Birch and Swinnerton-Dyer conjecture for $A$, which asserts that the Tate--Shafarevich group of $A$ is finite, that the $L$-function of $A$ vanishes at $s = 1$ to order $\rk A(K)$, and that the leading Taylor coefficient of the $L$-function satisfies the expected formula. 
    \par
    Let us collect a few other observations about the arithmetic of $A$, which is largely governed by the monodromy of the corresponding Lefschetz pencil $f$.
    \begin{prop}\label{propArithA}
        The following are true.
        \begin{enumerate}[(i)]
            \item The $L$-function of $A$ is
            \begin{gather*}
                L(A, s) = \frac{\chi(H^4(Y_{\ol{k}}, \Q_\ell(1)), s)}{(1-q^{1-s})^2}.
            \end{gather*}
            \item The base change of $A$ to $\ol{K}$ has endomorphism algebra isomorphic to $\mathbb{Q}$, and is not isogenous to an abelian variety defined over $\ol{k}$.
            \item The conductor of $A$ is the sum in $\Div(\mathbb{P}^1_k)$ of the points over which $f: Y \to \mathbb{P}^1_k$ is not smooth. In particular, $A$ has everywhere semistable reduction.
        \end{enumerate}
    \end{prop}
    \begin{proof}
        Claim (i) is immediate from proposition \ref{propLefschetzL} and the Betti numbers of cubic hypersurfaces. The \enquote{big} geometric monodromy of the Tate module of $A$ (propositions \ref{propVanSpan} and \ref{propOpenImage}), together with the Tate conjecture for endomorphisms of abelian varieties (\cite{Tate1966}, \cite{Faltings1986}, \cite{Zarhin1974a}, \cite{Zarhin1974b}, \cite{MB1985}), yields (ii).
        \par
        Finally, the conductor of $A$ is computed directly from proposition \ref{propInertia}, as follows. Fix $t \in \mathbb{P}^1_{\F_q}$ with $Y_t$ singular, and $\ol{t} \to t$ a geometric point. Since the computation of the conductor exponent at $t$ is entirely local, we will elide any discussion of paths, and simply fix a vanishing cycle $\delta \in H^3(Y_{\ol{\eta}}, \Q_\ell(1))$ for $\ol{t}$. Propositions \ref{propVanSpan} and \ref{propVanTran} imply that $\delta$ is nonzero. The inertia $I_{\ol{t}}$ acts tamely on the rational $\ell$-adic Tate module $V_\ell(A) \cong H^3(Y_{\ol{\eta}}, \Q_\ell(1))^\vee$, by proposition \ref{propInertia}, and the inertia-invariant subspace is the orthogonal complement of $\delta$, so the conductor exponent of $V_\ell(A)$ at $t$ is
        \begin{gather*}
            \dim_{\Q_\ell} V_\ell(A) - \dim_{\Q_\ell} V_\ell(A)^{I_{\ol{t}}} = 1.
        \end{gather*}
    \end{proof}
    The Birch and Swinnerton-Dyer conjecture in positive characteristic is known for isotrivial abelian varieties \cite[Theorem 3]{Milne1968} and for abelian varieties of analytic rank zero (for Jacobians over global function fields, combining \cite[Formula (2.9.3)]{Ulmer2011} with the arguments of \cite[Section 4]{Tate1966b} shows that analytic rank upper bounds algebraic rank; the extension of this inequality to general abelian varieties is well-known folklore). The above proposition thus indicates that the conjecture for $A$ is \enquote{hard}, in that it is not implied by either of these criteria. This may be unsurprising, since we used the highly nontrivial Tate conjecture for cubic fourfolds in order to prove it.

    \begin{exmp}\label{exmpExplicitL}
        We explicitly compute the $L$-function of $A$ when $X$ is the Fermat cubic fourfold. Let $\Delta$ be the base locus of $f$, and for simplicity assume that $\Delta$ is isomorphic to a blowup of $\mathbb{P}^2$ at six points (which can be guaranteed after passing to a finite extension of $\F_q$). Using the blowup formula for motives and proposition \ref{propArithA}, the $L$-function of $A$ can now be calculated directly from the zeta function of $X$. The zeta functions of Fermat varieties were expressed by Weil in terms of Jacobi sums \cite[Section 2.3]{KS79}. In particular, when $p \equiv 2$ modulo 3, $H^4(X_{\ol{\F}_q}, \Q_\ell(2))$ is generated by algebraic cycles defined over $\F_q$ (loc. cit., Theorem 2.10), so that
        \begin{align*}
            \chi(H^4(X_{\ol{\F}_q}, \Q_\ell(1)), s) &= (1-q^{1-s})^{23},\\
            \chi(H^4(Y_{\ol{\F}_q}, \Q_\ell(1)), s) &= (1-q^{1-s})^{30},\\
            L(A, s) &= (1-q^{1-s})^{28}.
        \end{align*}
        In this case, the Taylor expansion of the $L$-function at $s = 1$ is
        \begin{gather*}
            L(A, s) = (\ln(q) (s-1))^{28} + O((s-1)^{29}),
        \end{gather*}
        and the Mordell--Weil rank of $A$ is 28. The leading Taylor coefficient suggests that the BSD invariants of this $A$ are as simple as possible.
    \end{exmp}

    The results presented thus far in this section have analogues in higher dimension. Let $V$ be a smooth cubic sixfold over $k$, and let $g: W \to \mathbb{P}^1_k$ be a Lefschetz pencil for $V$. Then $W_{\eta}$ is a smooth cubic fivefold over $K$. By corollary \ref{corH5Cubic}, there exists an abelian variety $B$ over $K$ such that $\mf{h}^5(W_\eta) \cong \mf{h}^1(B)(-2)$. Proposition \ref{propFd} thus implies that $\BB^3(W_\eta)$ is equivalent to the Birch and Swinnerton-Dyer conjecture for $B$. In particular, if $V$ satisfies the hypotheses of theorem $\ref{thmBB3Cubic}$, then BSD holds for $B$.
    \par
    The proof of the next proposition is identical to that of proposition \ref{propArithA}, and omitted.
    \begin{prop}\label{propArithB}
        The following are true.
        \begin{enumerate}[(i)]
            \item The $L$-function of $B$ is
            \begin{gather}\label{eqLB}
                L(B, s) = \frac{\chi(H^6(W_{\ol{k}}, \Q_\ell(2)), s)}{(1-q^{1-s})^2}.
            \end{gather}
            \item The base change of $B$ to $\ol{K}$ has endomorphism algebra isomorphic to $\mathbb{Q}$, and is not isogenous to an abelian variety defined over $\ol{k}$.
            \item The conductor of $B$ is the sum in $\Div(\mathbb{P}^1_k)$ of the points over which $g: W \to \mathbb{P}^1_k$ is not smooth. In particular, $B$ has everywhere semistable reduction.
        \end{enumerate}
    \end{prop}
    \begin{exmp}
        Suppose that $V$ is the Fermat cubic sixfold, and $\character(k) \equiv 2 \bmod 3$. Then $V$ satisfies the hypotheses of theorem \ref{thmBB3Cubic}, and the cohomology of $V$ is generated by algebraic cycles defined over $\F_q$ \cite[Theorem 2.10]{KS1983}. Explicitly computing the $L$-function of $B$ in this case is less straightforward than in example \ref{exmpExplicitL}, because $W$ is the blowup of $V$ along a smooth cubic fourfold $Q$, whose contribution to the numerator of \eqref{eqLB} is nontrivial to write down. Nevertheless, we have 
        \begin{gather*}
            \dim_{\Q_\ell} H^6(V_{\ol{k}}, \Q_\ell(2)) + 1 = 88 \le  \rk B(K) + 2 \le \dim_{\Q_\ell} H^6(W_{\ol{k}}, \Q_\ell(2)) = 110,
        \end{gather*}
        where the $+ \ 1$ on the left comes from the self-intersection of an ample divisor on $Q$.
    \end{exmp}
    
    We conclude by proving a case of the Tate conjecture. Let $F$ be the Fano surface of $Y_\eta$, let $\mc{F} \to \mathbb{P}^1_k$ be a flat projective morphism with generic fiber $F$, and let $\rho: T \to \mc{F}$ be a resolution of singularities of $\mc{F}$, such that $\rho$ is a projective morphism. Such a resolution exists by \cite[Theorem 2.1]{CP2008}.

    \begin{thm}\label{thmTate}
        The Tate conjecture holds for $T$.
    \end{thm}
    \begin{proof}
        The Tate conjecture holds for $F$ (see theorem \ref{thmFanoAbType}), and the Birch and Swinnerton-Dyer conjecture holds for $A = \Alb(F)$, so the Tate conjecture holds for divisors on $T$ by \cite[Theorem 1.1]{Geisser2021}. The hard Lefschetz theorem \cite[Th\'eor\`eme 4.1.1]{Deligne1980} in turn implies the full Tate conjecture for $T$.
    \end{proof}
    
\printbibliography
\end{document}